\documentclass[preprint,10pt]{elsarticle}

\usepackage{amsmath}
\usepackage{url}

\def\OO{\mathcal{O}}

\def\ff{\frak}
\def\Spec{\mbox{\rm Spec}}
\def\Kr{{\rm Kr}}

\def\Max{\mbox{\rm Max}}

\def\patch{\mbox{\rm patch}}
\def\gen{\mbox{\rm gen}}

\def\inv{{\rm inv}}

\def\pt{\mbox{\rm pt}}

\def\cl{\mbox{\rm cl}}
\def\cal{\mathcal}

\def\X{{\ff X}}

\usepackage{amssymb,amscd}
 \usepackage{amsthm}

\begin{document}

\begin{frontmatter}



\title{Affine schemes and topological closures in the Zariski-Riemann space of valuation rings}

%over projective  schemes}


\author{Bruce Olberding}

\address{Department of Mathematical Sciences, New Mexico State University,
Las Cruces, NM 88003-8001}

\ead{olberdin@nmsu.edu}

\begin{abstract}
%For a subring $D$ of a field $F$, the Zariski-Riemann space of $F/D$ is the set of all valuation rings containing $D$ and having quotient field $F$.  The Zariski-Riemann space admits a natural topology that is compatible with the geometry of projective integral schemes.  We describe connections between the topology and geometry of this space, as well as how this connection manifests in  the Zariski-Riemann space when it is viewed as a locally ringed space. 
     Let $F$ be a field, let $D$ be a subring of $F$, and let ${\ff X}$ be the Zariski-Riemann space of valuation rings containing $D$ and having quotient field $F$. 
We consider  the Zariski, inverse and patch topologies on ${\ff X}$   when viewed as a projective limit of projective integral schemes having function field contained in $F$, and we characterize the locally ringed subspaces of $\X$ that are affine schemes.   
\end{abstract}


\end{frontmatter}

\newtheorem{theorem}{Theorem}[section]
\newtheorem{lemma}[theorem]{Lemma}
\newtheorem{proposition}[theorem]{Proposition}
\newtheorem{corollary}[theorem]{Corollary}
\newtheorem{example}[theorem]{Example}
\newtheorem{remark}[theorem]{Remark}
\newtheorem{definition}[theorem]{Definition}

%%%%%%%%%%%%%%%%%%%%%%%%%%%%%%%%%%%%\end{frontmatter}

%%%%%%%%%%%%%%%%%%%%%%%%%%%%%%%%%%%%%%%%%%%%%%%%%%%%%%%%%%%%%%%%%
%%%%%%%%%%%%%%%%%%%%%%%%%%%%%%%%%%%%%%%%%%%%%%%%%%%%%%%%%%%%%%%%%
\section{Introduction}

Throughout this article, $F$ is a field, $D$ is a subring of $F$ and  
 ${\ff X}$ denotes the collection of all valuation rings between $D$ and $F$ having quotient field $F$.  For each subset $S$ of $F$, denote by $\X_S$ the set of all valuation rings in ${\ff X}$ containing $S$.  
 The set ${\ff X}$ is  endowed with the topology whose basic open sets are of the form $\X_S$, where $S$  is a finite subset of $F$.  The topological space ${\ff X}$ is the
{\it Zariski-Riemann space} of $F/D$ (sometimes called the generalized Riemann manifold, Riemann-Zariski space, Zariski-Riemann manifold, abstract Riemann surface,  or  Riemann variety), and the topology thus defined is the {\it Zariski topology} on ${\ff X}$.  Zariski showed that ${\ff X}$ is quasicompact and used this as a step in the proof of  resolution of singularities of algebraic surfaces in characteristic $0$  by reducing  an infinite resolving system to a finite one \cite{ZarCom}.  But  ${\ff X}$ itself can be viewed as a geometric object.  It is a locally ringed space with structure sheaf $\OO_{\ff X}$ defined by 
 $\OO_{\X}({U})=\bigcap_{V \in {U}}V$ for each nonempty open subset ${U}$ of ${\ff X}$.  Moreover, as a locally ringed space, ${\ff X}$ is the projective limit of the projective models of $F/D$, those projective integral schemes over $\Spec(D)$ whose function field is a subfield of $F$.   Thus, while generally not a scheme, ${\ff X}$ is a projective  limit of projective schemes, and the valuation rings in ${\ff X}$ can be  used to track generic points of closed subschemes in blowups of projective models of $F/D$; this is the point of view taken in Zariski's theory of birational correspondence \cite{ZarBir}. 
%Other sophisticated applications, all but the first of  which is  recent, include 
%embedding a  separated scheme of finite type over a Noetherian base scheme as an open subscheme of a proper scheme  (Nagata \cite{NagCom1,NagCom2}); algebraic flattening (Fujiwara-Kato \cite{FujiKato}); factorization of a separated morphism of quasicompact quasiseparated schemes into a composition of an affine morphism and a morphism of finite presentation (Temkin \cite{Tem}); sequential compactness in analytic spaces (Favre \cite{Fav}); and degree growth of iterates of meromorphic  self-maps of compact K\"ahler surfaces (Boucksom, Favre and Jonsson \cite{BFJ}).  

The aim of this article is to   
  develop from a basic point of view some of the topological   features   of the Zariski-Riemann space as a projective limit of projective models.   We often do not assume anything more about $D$ other than that it  is a subring of $F$, and sometimes that $\Spec(D)$ is a Noetherian space, so as to assure that projective models of $F/D$ are Noetherian spectral spaces.  
The focus  is the interplay between the topology of $\X$ and that of the projective models of $F/D$.  
 By the ``topology'' of a subspace $Z$ of  ${\ff X}$ we mean not only the  topology on $Z$  induced by the Zariski topology, but the patch and inverse topologies also.  In this way, our approach to the topological nature of $\X$ is 
  influenced by the article \cite{FFL}, where the authors consider these three topologies also.  
  Whereas in \cite{FFL} the authors use the patch topology (and a  nice interpretation of the patch topology in terms of ultrafilter limits of valuation rings; see Remark~\ref{uf remark}(1) below) as a unifying theme, our preference is for the inverse topology because of its application to classifying affine schemes in $\X$ (see Section 6), as well as for describing irredundant representations, as is done in \cite{OlbTopIrr}.  However, as noted in Proposition~\ref{spectral basics} and in \cite[Remark 2.2]{FFL}, inverse closure is simply the composition of patch closure with closure under generalizations.    
  As in \cite{FFL}, we also emphasize the Kronecker function ring construction from multiplicative ideal theory for representing ${\ff X}$ as the prime spectrum of a ring.   
In particular,  in Section 4 we use this construction to exhibit  an affine scheme that maps onto ${\ff X}$ via a morphism of locally ringed spaces that is, by a theorem of Dobbs and Fontana, a homeomorphism on the underlying spaces.  The ring of global sections of this scheme is a Pr\"ufer domain that encodes the valuation theory of ${\ff X}$ into the prime spectrum of a ring. In particular, as proved in \cite{DFF} and \cite{DF}, $\X$ is a spectral space. 
  
 In Section 2 we develop some basic properties regarding patch closure and inverse closure in a spectral space.  In 
 Section 3, we discuss Zariski's representation of $\X$ as a projective limit of projective models. One of the main consequences that we draw from this is that when the base ring $D$ is Noetherian, then $\X$ is a projective limit of Noetherian spectral spaces. As discussed in Section 2, the patch and inverse closures are more transparent on Noetherian spectral spaces, so this representation of $\X$ as a projective limit is helpful  for clarifying the patch and inverse closures; in particular, in Corollary~\ref{fundamental cor}, we show how patch and inverse closures are determined by the images of the domination maps from $\X$ to the approximating projective models. As an application, we describe some patch dense subsets of $\X$. 
  In Section 5, we take a different point of view and describe the inverse closure of a subset of $\X$ in terms of its image in the affine scheme represented by the prime spectrum of the Kronecker function ring of $F/D$.  This allows us to characterize in Theorem~\ref{affine scheme} the subspaces of $\X$ that are affine schemes when viewed with the locally ringed space structure    inherited from $\X$.  
  
  While this article is focused on the topology of 
the Zariski-Riemann space, in future articles we apply these ideas to describing  properties  of irredundant intersections of valuation rings and the geometry of integrally closed rings.

\medskip

{\it Conventions.}  All rings are commutative and contain an identity.  
  An {\it overring} of a domain $R$  is a ring between $R$ and its quotient field.  The set of prime ideals of the ring $R$ is denoted $\Spec(R)$; the set of maximal ideals by $\Max(R)$.  When relevant, $\Spec(R)$ denotes not only a set but  an  affine scheme, and sometimes a submodel of a projective model, but these different uses should always be clear from context.    
As noted above, we  
 write ${\ff X}_R$ for $\{V \in {\ff X}: R \subseteq V\}$, so that $\X_R$ is the subspace of ${\ff X}$ consisting of the valuation rings in ${\ff X}$ between $R$ and $F$.
As discussed in Section 3, we do not assume that a projective model of $F/D$ has function field $F$.  Explanations of other  variations on traditional terminology can be found in the following places:  
generic point of a not-necessarily-closed subset of a spectral space (after Corollary~\ref{very new cor});
affine subset of $\X$ (beginning of \S 4); 
dominant system of projective models (end of \S 3); $X(Z) = $ the image of a subset $Z$ of $\X$ in the projective model $X$ (\S 3). 

\section{Spectral spaces}

A topological space $X$ is a {\it spectral space} if $X$ is quasicompact and $T_0$;  the quasicompact open 
subsets of $X$ are closed under finite intersection and form an open basis; and every 
nonempty irreducible closed subset of $X$ has a generic point.  By a theorem of Hochster, these are precisely the topological spaces which arise as the prime spectrum of a commutative ring \cite{Hoc}. As we recall in the next  section, ${\ff X}$ and the projective models of $F/D$ are spectral spaces. Thus the topological notions developed in this section for spectral spaces will apply to these two cases.  
Although we consider several topologies on a spectral space, to minimize confusion we introduce notation for  operators which we use to distinguish certain subsets of $X$.  Where possible we use these operators rather than shift between topologies. These operators are defined in terms of standard topological notions on $X$.
% with one exception: 
% We say a point $y \in Y$ is {\it sharply closed in $Y$} if there exists a quasicompact open subset $U$ of $X$ such that $\{y\} = Y {\smallsetminus} {U}$; equivalently, $Y \smallsetminus \{y\}  \subseteq U$ with $y \not \in U$. 
 % The terminology is based on that of sharp prime ideals in multiplicative ideal theory \cite{FHLu}. 
\begin{eqnarray*}
\cl(Y) & = & {\mbox{intersection of all closed sets containing }} Y.\\
\inv(Y) & = & {\mbox{intersection of all quasicompact open sets containing }} Y. \\
\gen(Y) & = & {\mbox{intersection of all  open sets containing }} Y. \\
\patch(Y) & = & {\mbox{intersection of all subsets between }} Y {\mbox{ and }}  X {\mbox{ of the form }} U_1 \cup (X {\smallsetminus} U_2), \\
&  \: & 
 {\mbox{where }} U_1, U_2 {\mbox{ are  quasicompact open sets of }} X.\\
\pt(Y) & = &  {\mbox{the points of }} \inv(Y)  {\mbox{ closed in the subspace topology}} . \\
%\sharp(Y) & = & {\mbox{the sharply closed points of }}Y.
%
%{\mbox{the inverse isolated points in }} Y.  
\end{eqnarray*}

Thus $\cl(Y)$ is the closure of $Y$ in $X$. 
 Both $\inv(Y)$ and $\patch(Y)$ can also be interpreted as closures in appropriate topologies; namely, the subset $Y$ of $X$ is 
{\it patch closed} if $Y = \patch(Y)$; $Y$ is {\it inverse closed } if $Y = \inv(Y)$.    
The {\it inverse topology} on the spectral space $X$ is the topology whose closed sets are the inverse closed subsets of $X$, while 
 the {\it patch topology} on $X$ has for closed sets the patch closed subsets of $X$.   The patch topology gives $X$ the structure 
 of   a zero-dimensional compact Hausdorff space \cite[p.~72]{John}.    The set $\gen(Y)$ is the closure of $Y$ under generalizations, and while it too defines a topology, we will not have occasion to use this topology.  (Recall that if $x,y \in X$ and $x \in \cl(\{y\})$, then $y$ is a {\it generalization} of $x$ and $x$ is a {\it specialization} of $y$.)  
 
 A subset $Y$ of $X$ is {\it inverse open} if its complement in $X$ is inverse closed; $Y$ is {\it patch open} if its complement in $X$ is patch closed.  
 %Thus $\sharp(Y)$ is the set of inverse open points in $Y$; that is, $\sharp(Y)$ is the set of points in $Y$ that are {\it isolated}  in the inverse topology.
It is clear that closed subsets of $X$ and  inverse closed subsets of $X$ are  patch closed, and hence   the patch topology refines  the spectral topology on $X$ and  its inverse topology.
   Both of these topologies were used by Hochster in \cite{Hoc}, who proved that the inverse and patch topologies are again spectral. (Hochster did not give a name to the inverse topology; this  terminology  was evidently introduced  by Schwartz in \cite{Sch}.) 
    More precisely,  let $Y$ and $X$ be spectral spaces. Then a map $f:Y\rightarrow X$ is {\it spectral}  if it is continuous and the preimages of quasicompact open subsets are quasicompact. 
    %i.e., $f$ is spectral if it is continuous in the spectral and inverse topologies on $X$.   
The patch closed subsets of $X$ are precisely those subsets $Y$ of $X$ that are spectral in the subspace topology and for which $Y \cap U$ is quasicompact in $Y$ for all quasicompact open subsets $U$ of $X$ (i.e., the inclusion mapping $Y \rightarrow X$ is a spectral map) \cite[p.~45]{Hoc}.

\begin{proposition} \label{spectral basics}
The following statements hold for  a subset $Y$ of the spectral space $X$.  
\begin{itemize}
\item[{\em (1)}]  $\inv(Y)$ and $\patch(Y)$ are spectral spaces in the subspace topology of $X$.
\item[{\em (2)}] $\gen(Y) \subseteq \inv(Y)$ and $\:  \pt(Y) \subseteq \patch(Y) \subseteq \inv(Y) \cap \cl(Y)$.

\item[{\em (3)}]  $\inv(Y) = \gen(\pt(Y)) = \gen(\patch(Y))$.  
\end{itemize}

\end{proposition}  
 
 \begin{proof} 
 That $\patch(Y)$ is a spectral space in the subspace topology is discussed above.    Since an inverse closed subset is patch closed, it follows that $\inv(Y)$ is also a spectral space in the subspace topology. 
 To prove (2), first observe that it is 
clear from the relevant definitions that $\gen(Y) \subseteq  \inv(Y)$ and  $\patch(Y) \subseteq \inv(Y) \cap \cl(Y)$.   
  We  claim  that 
  $\pt(Y) \subseteq \patch(Y)$.  Since $\patch(Y) \subseteq \inv(Y)$ and $\inv(Y)$ is patch closed, it suffices to show that $\pt(Y)$ is a subset of the patch closure of $Y$ in $\inv(Y)$.  
Let ${U}$ be a quasicompact open subset of $\inv(Y)$, and let ${V}$ be the complemement of a quasicompact open subset of $\inv(Y)$.  Since $\inv(Y)$ is quasicompact  (this follows from (1)), so  
 is the  closed subset ${V}$, and hence ${U} \cup {V}$, as a union of two quasicompact subsets, is quasicompact.  Since $Y \subseteq {U} \cup {V} \subseteq \inv(Y)$, it follows that ${U} \cup {V}$ meets every inverse open subset of $\inv(Y)$.  Each  $x \in \pt(Y)$ is a  closed point in $\inv(Y)$, so $\{x\}$ is an intersection of complements of quasicompact open subsets of $\inv(Y)$.  
   However, the intersection of finitely many of any of these inverse open (and spectral closed) subsets meets ${U} \cup {V}$, so since ${U} \cup {V}$ is quasicompact, this forces $x \in {U} \cup {V}$, proving that $\pt(Y) \subseteq \patch(Y)$. This proves (2).

To prove (3), note that since $\inv(Y)$ is a spectral space, then for each $y \in \inv(Y)$, 
   there exists a closed point $x \in \inv(Y)$ such that $x \in \cl(\{y\})$.  Hence  $\inv(Y) \subseteq \gen(\pt(Y))$, and since the reverse inclusion is clear, $\inv(Y) = \gen(\pt(Y))$. Moreover, since by (2), $\pt(Y) \subseteq \patch(Y) \subseteq \inv(Y)$, it follows then that $\inv(Y) = \gen(\patch(Y))$.  
 \end{proof} 
 
 That $\inv(Y) = \gen(\patch(Y))$  is also noted in \cite[Remark 2.2]{FFL}.

%\cite[Corollary, p.~45]{Hoc}.  
%If $Y$ is patch closed, then 
%$\cl(Y) = \bigcup \{\cl(y):y \in Y\}$.    

%In Proposition~\ref{collect 3}, we apply the following proposition to Noetherian subspaces of spectral spaces.    

%    A {\it spectral map} $f:Y \rightarrow X$ between spectral spaces is a continuous map such that the  preimages of quasicompact open subspaces are quasicompact.

If $Y$ is a subspace of the spectral space $X$ and not every point in  $Y$ is closed, then $Y \not \subseteq \pt(Y)$. However,  when $Y$ is quasicompact, then $\pt(Y) \subseteq Y$, and also in this  case,  $\inv(Y)$ collapses to  $\gen(Y)$:    
 
\begin{proposition} \label{new qc} The following are equivalent for a subspace $Y$ of the spectral space $X$.  
\begin{itemize}
\item[{\em (1)}]  $Y$ is quasicompact.
\item[{\em (2)}]  $\inv(Y) = \gen(Y)$.
\item[{\em (3)}] $\pt(Y) \subseteq Y$.  
%If $Y$ is a quasicompact subspace of the spectral space $X$, then 
%$$\inv(Y) = \gen(Y) \: \: {\mbox{ and }} \: \:   
%\pt(Y) \subseteq Y.$$
\end{itemize}
\end{proposition}

\begin{proof}
(1) $\Rightarrow$ (2)  
By Proposition~\ref{spectral basics}(2), $\gen(Y) \subseteq \inv(Y)$.     Suppose  $x \in \inv(Y) {\smallsetminus} \gen(Y)$.  Then $\cl(\{x\})$  does not meet $Y$, so that $Y$ is an open subset of the quasicompact space $Y \cup \{x\}$. As such, $Y = \bigcup_{\alpha}((Y \cup \{x\}) \cap U_\alpha)$, where $\{U_\alpha\}$ is a collection of quasicompact open subsets of $X$.  Since $Y$ is quasicompact, it follows that 
$Y = (Y \cup \{x\}) \cap U$, where $U$ is a quasicompact open subset of $X$.  But then $Y \subseteq U$ and $x \not \in U$, in contradiction to the assumption that $x \in \inv(Y)$.  Thus $\inv(Y) = \gen(Y)$.  

(2) $\Rightarrow$ (3) Since by (2),  
 $\inv(Y) \subseteq \gen(Y)$, it follows that if $x \in \pt(Y)$, then there exists $y \in Y$ such that $y \in \cl(\{x\})$.  But since $x$ is a closed point  in $\inv(Y)$, this forces $x = y \in Y$.
 
 (3) $\Rightarrow$ (1)  Assuming (3), Proposition~\ref{spectral basics}(2) implies $\inv(Y) = \gen(\pt(Y)) \subseteq \gen(Y) \subseteq \inv(Y)$, and hence $\inv(Y) = \gen(Y)$.  Schwartz and Tressl show in  \cite[Proposition 2.3]{ST} that $Y$ is quasicompact if and only if $\gen(Y)$ is patch closed.  Thus since $\inv(Y)$ is patch closed, we conclude that $\gen(Y)$ is patch closed, and hence $Y$ is quasicompact.
\end{proof}

\begin{corollary} \label{very new cor} A spectral space $X$ is Noetherian if and only if $\inv(Y) = \gen(Y)$ for all (open) subsets $Y$ of $X$.  
\end{corollary} 

\begin{proof}  If $X$ is Noetherian, then so is any subspace $Y$ of $X$, and hence $Y$ is quasicompact and by the proposition satisfies $\inv(Y) = \gen(Y)$.  Conversely, if $\inv(Y) = \gen(Y)$ for all open subsets $Y$ of $X$, then by the proposition, each open subset is quasicompact.  This then implies that $X$ is a Noetherian space.  
\end{proof}

\begin{remark} {\em  Let $Y$ be a subspace of the spectral space $X$.  
The fact that  $\inv(Y) = \gen(Y)$ when $Y$ is  quasicompact   is a slightly stronger version of \cite[Proposition 2.6]{FFL} and \cite[Proposition 2.3]{ST}, where it is shown that the generic closure of a quasicompact set is patch closed. 
 A subspace $Y$ of $X$ is quasicompact in the inverse topology if and only if $\cl(Y) = \bigcup_{y \in Y} \cl(\{y\})$ \cite[Corollary 2.4]{ST}.  
}
\end{remark}

In Section~\ref{inverse section} we show that an inverse closed subspace of $\X$ admits the structure of a locally ringed space, a fact which ultimately depends on the following proposition.  

\begin{proposition} \label{inverse lrs}  An inverse closed subspace of an irreducible spectral space is irreducible. 
\end{proposition}

\begin{proof}  Let $Y$ be an inverse closed subspace of the irreducible spectral space $X$.   Let $U$ and $V$ be open sets of $X$ such that $U \cap Y$ and $V \cap Y$ are nonempty. Since ${X}$ is a spectral space, ${X}$ has a basis of quasicompact open subsets, so we may assume that $U$ and $V$ are quasicompact open subsets of ${X}$.   Since $Y$ is inverse closed in ${X}$, there is a collection $\{U_\alpha\}$ of quasicompact open subsets of ${X}$ such that $Y = \bigcap_{\alpha} U_{\alpha}$.  Since ${X}$ is irreducible, the intersection of any two nonempty open subsets in ${X}$ is nonempty, and hence the collection $\{U,V\} \cup \{U_\alpha\}$ has the finite intersection property.  Since a spectral space with the inverse topology is again a spectral space, then ${X}$ is quasicompact in the inverse topology. Therefore, the intersection of inverse closed sets $Y = U \cap V \cap (\bigcap_{\alpha}U_\alpha)$ is nonempty, which proves that $Y$ is irreducible. 
\end{proof}

When $Y$ is  a Noetherian subspace of $X$, then  
  the patch closure can be interpreted in terms of generic points of subsets of $Y$.  
   We say that $x$ is a {\it generic point} for a subset $Y$ of $X$ if $\cl(\{x\}) = \cl(Y)$; i.e., $x$ is the generic point of the closed set $\cl(Y)$.  
Recall that a subset  $Y$ of a topological space $X$ is {\it retrocompact} if $Y \cap U$ is quasicompact for every quasicompact open subset $U$ of $X$ (i.e., the inclusion map $Y \rightarrow X$ is a spectral map).

\begin{proposition} \label{patch Noetherian} Let $Y$ be a subset of the spectral space $X$.  

\begin{itemize}
\item[{\em (1)}]  If $x \in X$ and $x$ is a generic point for some subset of $Y$, then $x \in \patch(Y)$. 

\item[{\em (2)}]  If  $Y$ is retrocompact (which is the case if $Y$ is a Noetherian subspace of $X$), then $$\patch(Y) = \{x \in X:x {\mbox{ is a generic point for some subset of }} Y\}.$$  
\end{itemize}
\end{proposition}

\begin{proof}
 (1) Suppose that $x \in X$ and  $\cl(\{x\}) = \cl({Z})$ for some subset ${Z}$ of $Y$.  We claim that $ x\in \patch(Y)$.   Suppose otherwise.  Then there exists a quasicompact open subset $U$ of $X$, and a subset $V$ of $X$ that is the complement of a quasicompact open subset of $X$, such that $Y \subseteq U \cup V$ but $x \not \in U \cup V$.  If $Z \cap U$ is nonempty, then for any $z$ in this intersection,  $z \in Z \subseteq  \cl(\{x\})$, so that since $z$ is in the open set $U$, so is $x$, a contradiction.  Thus $Z \cap U$ is empty, and hence $Z  \subseteq  V$, so that $x \in \cl(Z) \subseteq V$, again a contradiction.  Therefore, $x \in \patch(Y)$. 
    
(2) 
Suppose that $Y$ is a retrocompact subspace of $X$.     In light of (1), to prove (2) all that needs to be shown is that if $x \in \patch(Y)$, then $x$ is a generic point for some subset of $Y$.  
Let $x \in \patch(Y)$. Then $\cl(\cl(\{x\}) \cap Y) \subseteq \cl(\{x\})$.   We claim that this set inclusion is an equality.  By way of contradiction, 
 suppose this  inclusion is proper, and let $U =  (Y \cup \{x\}) \smallsetminus \cl(\{x\})$ and $V =  \cl(\cl(\{x\}) \cap Y)$.  
Then $Y \subseteq U \cup V$ but $x \not \in U \cup V.$  Since the closed set $V$ is an intersection of complements of quasicompact open subsets of $X$ and $x \not \in V$, there exists a set $V'$ such that $x \not \in V'$, $V \subseteq V'$ and $V'$ is the complement of a quasicompact open subset of $X$. 
Also, since $Y$ is retrocompact,  $Y {\smallsetminus} V' = Y \cap (X {\smallsetminus} V')$ is quasicompact, so since $Y {\smallsetminus} V' \subseteq U$ and $U$ is a union of quasicompact open subsets of $X$, it follows that  there exists a quasicompact open subset $U'$ such that $Y {\smallsetminus} V' \subseteq U' \subseteq U$.    Now $Y \subseteq U' \cup V'$, and since $U' \subseteq U$, then $x \not \in U'$, and by the choice of $V'$, $x \not \in V'$.  Thus $Y \subseteq U' \cup V'$ but $x \not \in U' \cup V'$, which is impossible since $x \in \patch(Y)$.  This shows that every element of $\patch(Y)$ is a generic point for some subset of $Y$.    
\end{proof}

%\begin{remark} {\em When $X = \Spec(R)$ for a ring $R$, and $Y$ is a subset of $X$, then a prime ideal $P$ of $R$ is a generic point for $Y$ if and only if $P$ is the intersection of the  prime ideals in $Y$.  Thus when $Y$ is a Noetherian subspace of $X$, then by the corollary,  $\patch(Y)$ is the set of all prime ideals in $R$ which are an intersection of prime ideals from $Y$.}
%\end{remark}

We next apply these notions to projective limits of Noetherian spectral spaces. Theorem~\ref{fundamental} is our main result in this direction, and it follows from two preliminary lemmas.

  \begin{lemma} \label{pick up extra} \label{spectral map cor}  Let $d:Y \rightarrow X$ be a spectral map between spectral spaces, and let $Z$ be a subspace of $Y$.   Then: 
 \begin{itemize}
 
 \item[{\em (1)}]     The mapping $d$ is continuous in the patch and inverse topologies.

 \item[{\em (2)}]  $d(\gen(Z)) \subseteq \gen(d(Z))$, $d(\inv(Z)) \subseteq \inv(d(Z)) $ and $d(\patch({{Z}})) = \patch(d({{Z}}))$. 
 
  \item[{\em (3)}]  If $d$ is a closed map, then  $d(\pt(Z)) = \pt(d(Z)).$

 \item[{\em (4)}] 
Each point in $X$ that is a generic point for a subset of $d({{Z}})$ is in  $d(\patch({{Z}}))$.
 
% \item[{\em (3)}]  [OMIT?]
% Each  point in $X$ that specializes to a point in $d(Z)$ is in $d(\inv(Z))$.
 
 \item[{\em (5)}] If also $X$ is a Noetherian spectral space, then $\pt(d(Z))$ is the set of points that are closed in the subspace $d(Z)$ of $X$. 
 \end{itemize}  
    
\end{lemma}  

\begin{proof} 
(1)  This follows from the relevant definitions.

(2) 
 %Similarly, $d(\inv(Z)) = \inv(d(Z))$.  
 To see that $d(\gen(Z)) \subseteq \gen(d(Z))$, let $x \in \gen(Z)$. Then there exists $z \in Z$ such that $z \in \cl(\{x\})$. Hence  $d(z) \in d(\cl(\{x\})) \subseteq \cl(\{d(x)\})$, so that $d(x) \in \gen(d(z))$. This shows that $d(\gen(Z)) \subseteq \gen(d(Z))$.  
Also, since $d$ is continuous in the inverse topology, it follows that 
  $d(\inv(Z)) \subseteq \inv(d(Z))$.  Finally, since the patch topology is compact and Hausdorff and a continuous map between compact Hausdorff spaces is closed, the mapping $d$ is closed in the patch topology, and hence $d(\patch({{Z}})) = \patch(d({{Z}}))$.

(3)  Suppose that $d$ is a closed map. 
 To show that $d(\pt(Z)) \subseteq \pt(d(Z))$, let $x \in d(\pt(Z))$.  Then there exists $y \in \pt(Z)$ such that $d(y) = x$.
 By Proposition~\ref{spectral basics}(2), $\pt(Z) \subseteq \patch(Z)$, so $y$ is a closed point in $\patch(Z)$.  
   Since $d$ is a closed map and $y$ is a closed point  in $\patch(Z)$, then $d(y)$ is a closed  in $d(\patch(Z)) = \patch(d(Z))$. It follows then that $d(y)$ is a closed point in $\inv(d(Z))$,  
      and hence $x \in \pt(d(Z))$.  
      
        Conversely, if $x \in \pt(d(Z))$, then $x$ is a closed point  in $\inv(d(Z))$, and hence a closed point in $\patch(d(Z))$. By Lemma~\ref{pick up extra}(2), $\patch(d(Z)) = d(\patch(Z))$, so $x$ is a closed point in $d(\patch(z))$. 
          As such, $d^{-1}(x) \cap \patch(Z)$ is a closed set in $\patch(Z)$.  Since by Proposition~\ref{spectral basics}(1), $\patch(Z)$, and hence $d^{-1}(x) \cap \patch(Z)$, is a spectral space, there exists $y \in \patch(Z)$ such that $y$ is a closed point in the closed subset $d^{-1}(x) \cap \patch(Z)$ of $\patch(Z)$.  Thus $y$ is a closed point in $\patch(Z)$ with $d(y) = x$. Since by Proposition~\ref{spectral basics}(3),  $\inv(Z) = \gen(\patch(Z))$, it follows that $y$ is a closed point in $\inv(Z)$, so that $y \in \pt(Z)$. Therefore, $x = d(y) \in 
           d(\pt(Z))$, which  proves that $\pt(d(Z)) = d(\pt(Z))$.  

(4)  Suppose there exists a subset $Y$ of $Z$ such that $x \in X$ is a generic point for $d(Y)$. Then by Proposition~\ref{patch Noetherian}(1),  $x \in \patch(d(Y))$, and hence by (2), $x \in d(\patch(Y))$.   

(5) 
Let $x$ be a closed point in $ \inv(d(Z))$.  Then  since by assumption $X$ is a Noetherian spectral space, Proposition~\ref{new qc} implies that  $x \in \pt(d(Z)) \subseteq d(Z)$.  Thus $x$ is a closed point in $d(Z)$. Conversely, let $y$ be a closed point in $d(Z)$. Then $y \in \inv(d(Z))$, so  by Proposition~\ref{spectral basics}, since $\inv(d(Z))$ is a spectral space there exists a closed point $w \in \inv(d(Z))$ such that $w \in \cl(\{y\})$.  The preceding argument  shows that $w$ is a closed point in $d(Z)$.  But $y$ is a closed point in $d(Z)$, so since $w \in \cl(\{y\}) \cap d(Z)$, it follows that $y = w$ and $y$  is a closed point in $\inv(d(Z)))$; i.e., $y \in \pt(d(Z))$. 
\end{proof}

%\begin{lemma} \label{KHaus} Let $\{X_\alpha\}$ be a projective system of Hausdorff spaces such that $X = \varprojlim \:X_\alpha$ is a compact space. For each $\alpha$, let $d_\alpha:X \rightarrow X_\alpha$ denote the induced mapping. Then for each closed subset $Y$ of $X$, $Y = \bigcap_{\alpha} d_\alpha^{-1}(d_\alpha(Y))$.   
%\end{lemma}

%\begin{proof} It is clear that $Y \subseteq \bigcap_{\alpha} d_\alpha^{-1}(d_\alpha(Y))$. Let $x \in \bigcap_{\alpha} d_\alpha^{-1}(d_\alpha(Y))$. Then for each $\alpha$, $d_\alpha(x) \in d_{\alpha}(Y)$, and hence $Y \cap d_\alpha^{-1}(d_\alpha(x))$ is nonempty. Since $\{X_\alpha\}$ is a projective system, it follows that the  collection of all $Y \cap d_\alpha^{-1}(d_\alpha(x))$, where $\alpha$ ranges over the index set of $\{X_\alpha\}$, has the finite intersection property. Since $X$ is a Hausdorff space, each $d_\alpha(x)$ is closed in $X_\alpha$, so that $Y \cap d_\alpha^{-1}(d_\alpha(x))$ is a closed subset of $Y$. As a closed subspace of the compact space $X$, $Y$ is a compact space, so we conclude that there exists $y \in Y \cap (\bigcap_\alpha d_\alpha^{-1}(d_\alpha(x)))$. Therefore, $d_\alpha(y) = d_\alpha(x)$ for all $\alpha$, so that $x = y \in Y$, which completes the proof.
%\end{proof}

The operators  $\patch$ and $\inv$  also behave well with respect to  projective limits.

\begin{lemma} \label{patch thm} 
Let $Y$ be a spectral space that is a projective limit (in the category of spectral spaces with spectral maps) of a projective system ${\ff S}$ of spectral spaces. Let $Z$ be a subset of $Y$. For each $X \in {\ff S}$, let $d_X:Y  \rightarrow X$ denote the projection, and let $X(Z)=d_X(Z)$. 
 Then $\{\patch(X({{Z}})):X \in {\ff S}\}$ and $ \{\inv(X({{Z}})):X \in {\ff S}\}$ form  projective systems with the induced maps. 
As topological spaces, 
$$\patch({{Z}}) = \varprojlim\: \patch(X({{Z}})) \:\:  {\mbox{ and }}  \:\:\: \inv({{Z}}) = \varprojlim\: \inv(X({{Z}})),$$ where in each case $X$ ranges over the spaces in ${\ff S}$. 
%
% $$\patch({{Z}}) \cong \varprojlim\: \{\patch(X({{Z}})):X \in {\ff S}\}, \:\:   
% \pt(Z) \cong \varprojlim\: \{\pt(X(Z)):X \in {\ff S}\}$$ $${\mbox{ and }} \: \inv({{Z}}) \cong \varprojlim\: \{\inv(X({{Z}})):X \in {\ff S}\}.$$
%where $X$ ranges over ${\ff S}$.  
\end{lemma}

\begin{proof} If $d:X_1 \rightarrow X_2$ is a map in the projective system ${\ff S}$, then Lemma~\ref{pick up extra} implies that $d(\patch(X_1(Z))) = \patch(X_2(Z))$ and $d(\inv(X_1(Z))) \subseteq \inv(X_2(Z))$. 
It follows that $\{\patch(X({{Z}})):X \in {\ff S}\}$ and $ \{\inv(X({{Z}})):X \in {\ff S}\}$ form  projective systems with the induced maps.  
Now 
 $\patch(-)$ is a closure operator in the patch topology and the projection mappings are continuous in the patch topology (Lemma~\ref{pick up extra}(1)), so $\patch(Z)$ is the projective limit of the spectral spaces $\patch(X(Z))$  \cite[Corollary, p.~49]{BouTop}. Similarly, since $\inv(-)$ is a closure operator in the inverse topology and the projections mappings are continuous in the inverse topology (again by Lemma~\ref{pick up extra}(1)), then $\inv(Z)$ is the projective limit of the $\inv(X(Z))$.        
\end{proof}

The lemmas lead to the main theorem of the section, a characterization of $\inv$, $\patch$ and $\pt$  in the case where the spectral spaces in ${\ff S}$ are Noetherian. This theorem will be applied to the Zariski-Riemann space in Corollary~\ref{fundamental cor}. 

\begin{theorem}  \label{fundamental}
With the assumptions of Lemma~\ref{patch thm}, suppose in addition that each space in $ {\ff S}$ is  Noetherian. 
 Then:
\begin{eqnarray*}
 \inv(Z) & = &\{y \in Y:\forall X \in {\ff S}, \:\: d_X(y) {\mbox{ specializes to a point in }}X(Z)\} \\
%$ is the set of all $V  \in {\ff X}$ such that  for every  $X \in {\ff S}$, the center of $V$ in $X$ specializes to a point in $X(Z)$.  
 \patch(Z) &  = &\{y \in Y:\forall X \in {\ff S}, \:\: d_X(y)  {\mbox{ is  a generic point of a subset of }} X(Z)\} 
 \end{eqnarray*} If also each $d_X$ is a closed map, then 
% is the set of all $V \in {\ff X}$ such that for every $X \in {\ff S}$, $V$ is centered in $X$ on a generic point for a subset of $X(Z)$.
%$\patch(Z)$ is the set of all $V \in {\ff X}$ such that for every $X \in {\ff S}$, $V$ is centered in $X$ on a generic point for a subset of $X(Z)$.  
$$  \pt(Z)  =  \{y \in Y:\forall X \in {\ff S}, \:\: d_X(y) {\mbox{ is  a closed point in }} X(Z)\}.   $$ 

 \end{theorem}

\begin{proof}
To prove the first assertion, let $y \in \inv(Z)$, and let $X \in {\ff S}$.  Then by Lemma~\ref{pick up extra}(2),  $d_X(y) \in X(\inv(Z)) \subseteq \inv(X(Z))$. Since $X$ is Noetherian, so is every subspace of $X$, and hence $X(Z)$ is quasicompact.   Thus by Proposition~\ref{new qc}, $d_X(y) \in X(\inv(Z)) \subseteq  \inv(X(Z)) = \gen(X(Z))$, and hence
$d_X(y)$ specializes to a point in $X(Z)$.  
 Conversely, suppose that $y \in Y$ has the property that for each $X \in {\ff S}$, $d_X(y)$ specializes to a point in $X(Z)$.  Then $d_X(y) \in \gen(X(Z)) =\inv(X(Z))$ for each $X \in {\ff S}$. Thus   
    $y$ is in  $\varprojlim\inv(X(Z))$  \cite[Corollary, p.~49]{BouTop}, so that by Lemma~\ref{patch thm}, $y \in \inv(Z)$.   
 
 Next, let $y \in \patch(Z)$, and let $X \in {\ff S}$.  By Lemma~\ref{pick up extra}(2), $d_X(y) \in X(\patch(Z)) = \patch(X(Z))$, and hence by Proposition~\ref{patch Noetherian}, $d_X(y)$ is a generic point for a subset of $X(Z)$.  Conversely, let $y \in Y$ and suppose that for all $X \in {\ff S}$, $d_X(y)$ is a generic point of a subset of $X(Z)$.  Then by Proposition~\ref{patch Noetherian}, $d_X(y) \in \patch(X(Z))$.    As in the case of $\inv(Z)$, this implies that $y \in \patch(Z)$.   

To verify the assertion about $\pt(Z)$, let $y \in \pt(Z)$, and let $X \in {\ff S}$.  Then by Proposition~\ref{new qc} and Lemma~\ref{pick up extra}(3), $d_X(y) \in X(\pt(Z)) = \pt(X(Z)) \subseteq X(Z)$.  Hence $d_X(y)$ is a closed point in $X(Z)$.  Conversely, suppose that $y \in Y$ and for each $X \in {\ff S}$,  $d_X(y)$ is a closed point  of $X(Z)$.  
Since $X(Z)$ is quasicompact, $\gen(X(Z)) = \inv(X(Z))$, and hence $d_X(y)$ is  a closed point in $\inv(X(Z))$.  Since by Lemma~\ref{patch thm}, $\inv(Z) = \varprojlim \inv(X(Z))$, then it follows that $y$ is a closed point in $\inv(Z)$, and hence  
 $y \in \pt(Z)$.  
%
%Finally,   let $y \in \sharp(Z)$. Then $Z\smallsetminus \{y\}$ is inverse closed in $Z$, so  $y \not \in \inv(Z \smallsetminus \{y\})$.  Hence using the description of $\inv(Z)$ above, there exists a projective model $X \in {\ff S}$ such that the center of $y$ in $X$ does not specialize to a point in $X(Z \smallsetminus \{y\})$.  Conversely, if $y \in Y$ and there exists $X \in {\ff S}$ such that the center of $y$ in $X$ does not specialize to a point in $X(Z \smallsetminus \{y\})$.  Then from the expression for $\inv(Z)$, we deduce that $y \not \in \inv(Z \smallsetminus \{y\})$.  Thus $\{y\}$ is an inverse open set in $\inv(Z)$; equivalently, 
% $y \in \sharp(Z)$.   
\end{proof}

\section{The Zariski-Riemann space as a limit of projective models} \label{Zariski-Riemann section}

%In later sections we assume also that $\Spec(D)$ is a Noetherian space.  This assumption is important  later, but in this section we only use it at one point, in Proposition~\ref{Noetherian spectral}.  
%  Thus the discussion in this section, with the exception of the proposition, is valid over an arbitrary domain $D$.
A {\it projective model} $X$ of $F/D$ is a projective integral scheme over $\Spec(D)$ whose function field is a subfield of  $F$.  Because we wish for our framework to be flexible enough to handle the case in which $D$ is the prime subring of $F$, we depart from contemporary usage and follow Zariski-Samuel \cite[Chapter VI, \S 17]{ZS} in that we do not require a  projective model of $F/D$  to have function field $F$.  However,  
 using the notion of a dominant system of projective models mentioned below, we can reduce to this case when it is possible; that is, when $F$ is a finitely generated field extension of the quotient field of $D$.  
 
 Each projective model $X$ can be realized in a convenient way.     
  Namely,   there exist $f_1,\ldots,f_n \in F$  such that if for each $i$, $D_i = D[\frac{f_1}{f_i},\ldots,\frac{f_n}{f_i}]$, then $X =  \bigcup_{i}\Spec(D_i)$, where prime ideals  ${\ff p}_i$ of $D_i$ and ${\ff p}_j$ of $D_j$ are identified in $X$ whenever $(D_i)_{\ff p_i} = (D_j)_{\ff p_j}$.
%In this case, we say that $X$ is the {\it projective model of $F/D$ defined by} $f_1,\ldots,f_n$.  
%  Alternatively, following \cite[Chapter VI, \S 17]{ZS}, one can simply view $X$ as the subspace of the space of all quasilocal overrings consisting of localizations of the $D_i$ at prime ideals; that is, $X$ can be identified with  $$\{(D_i)_{\ff p_i}:i=1,\ldots,n, \: {\ff p}_i \in \Spec(D_i)\}.$$    We view  $X$ abstractly, as a projective scheme, since this gives us a convenient way to parameterize the collection of local rings $(D_i)_{\ff p_i}$.  
We
  denote the structure sheaf of $X$ by $\OO_X$.  
  %Thus for a point $x \in X$, the stalk $\OO_{X,x}$ is the local ring $(D_i)_{{\ff m}_{X,x}}$, where $x$ is in the subscheme of $X$ identified with  
% $\Spec(D_i)$ and ${\ff m}_{X,x}$ is the prime ideal in $\Spec(D_i)$  corresponding to $x$.  
For each nonempty open subset $U$ of $X$, we view the stalk $\OO_{X,x}$ as a subring of $F$, so that 
  $\OO_X(U) = \bigcap_{x \in U}\OO_{X,x}$.   The {\it Zariski topology} on $X$ has as a basis of open sets the sets of the form $\{x \in X:f_1,\ldots,f_n \in \OO_{X,x}\},$ where $f_1,\ldots,f_n \in F$.  
% Hence if $x \in X$, then   $\cl(\{x\}) = \{y \in X:\OO_{X,y} \subseteq \OO_{X,x}\}$; see \cite[pp.~115-116]{ZS}.  Each local ring $\OO_{X,x}$ has the same quotient field, namely $\OO_{X,\eta}$, where $\eta$ is the generic point of $X$, and by assumption, $\OO_{X,\eta} \subseteq F$.     
%The {\it dimension} of $X$ is the supremum of the Krull dimensions of $\OO_{X,x}$, $x \in X$.   

%The Zariski topology on $X$ is compatible with the Zariski topology of the prime spectra that cover it.  More precisely, 
%write $X = \bigcup_{i=1}^m\Spec(D_i)$, where the $D_i$ are finitely generated $D$-subalgebras of $F$, and let $U$ be a subset of $X$.  Then $U$ is open if and only if for each $i=1,\ldots,m$, $U \cap \Spec(D_i)$ is an open subset of $\Spec(D_i)$ \cite[Lemma 2, p.~116]{ZS}.  Thus every open subset $U$ of $X$ can be expressed as a union of open subsets $U_i$, where each $U_i$ is open in $\Spec(D_i)$.  
 
 A projective model $Y$ of $F/D$ {\it dominates} a projective model  $X$ of $F/D$ if for each $y \in Y$, there exists $x \in X$ such that $\OO_{X,x} \subseteq \OO_{Y,y}$ and ${\ff m}_{X,x} = {\ff m}_{Y,y} \cap \OO_{X,x}$.  Thus when $Y$ dominates $X$, there is a {\it domination morphism}  $\delta^Y_X =(d,d^\#):Y \rightarrow X$, where the map $d:Y \rightarrow X$ on the underlying topological spaces is given by $d(y) = x$ (with $x$  as above) and the sheaf morphism $d^\#:\OO_X \rightarrow d_*\OO_Y$ is defined on each nonempty open subset $U$ of $X$ by $d^\#(U):\OO_X(U) \rightarrow \OO_Y(d^{-1}(U)):s \mapsto s$.  
   The mapping $d$ is continuous and closed \cite[Lemma 5, p.~119]{ZS}, and  $\delta^Y_X$ is a dominant morphism of  schemes.    
   
     Given any two projective models $X$ and $Y$ of $F/D$, there exists a projective model ${Z}$ of $F/D$ (namely, $Z = X \times_{{\rm Spec}(D)} Y$) such that $Z$ dominates both $X$ and $Y$ \cite[Lemma 6, p.~119]{ZS}, and hence the set of projective models of $F/D$ equipped with the domination morphisms forms a projective  system, a fact we revisit in Proposition~\ref{lr}.    Similarly,   
    ${\ff X}$ dominates each projective model $X$ of $F/D$ in the sense that for each $x \in X$, there exists a valuation ring $V \in {\ff X}$ such that $\OO_{X,x} \subseteq V$ and ${\ff m}_{X,x} = {\ff M}_V \cap \OO_{X,x}$, where ${\ff M}_V$ is the maximal ideal of $V$ \cite[pp.~119-120]{ZS}.  We say that $V$ is {\it centered on} $x$.  
In fact, there is a morphim of locally ringed spaces $\delta^\X_X=(d,d^\#):\X \rightarrow X$, where $d(V) = x$ (with $x$ as above), and $d^\#:\OO_X \rightarrow d_*\OO_{\X}$ is defined on nonempty  open subsets $U$ of $X$ by $d^\#_U:\OO_X(U) \rightarrow \OO_{\X}(d^{-1}(U)):s \mapsto s$.      
     The mapping $d$ is  surjective, continuous and  closed  \cite[Lemma 4, p.~117]{ZS}.  
   Abstractly, the domination morphism $\delta^{\X}_X$ encodes the valuative criteria for properness \cite[Theorem II.4.7]{Hart}.

For lack of a reference, we note next that the underlying topological space of a projective model  is a  spectral space.    The proof of the proposition uses the fact that ${\ff X}$ is a spectral space, which we discuss more in Remark~\ref{after prop} and the next section.

\begin{proposition} \label{Noetherian spectral} The underlying topological spaces of a projective model $X$  of $F/D$ is a spectral space.  If also  $\Spec(D)$ is a Noetherian space, then $X$ is a Noetherian spectral  space. 
\end{proposition}

\begin{proof} Let $X$ be a projective model of $F/D$.  Then  $X$ is a $T_0$-space having a basis of quasicompact open subsets.  Moreover, as discussed above, the domination mapping $d:\X \rightarrow X$ of the underlying  topological spaces   is a continuous surjective mapping, and, as noted below in Remark~\ref{after prop}, 
 ${\ff X}$ is a spectral space. Thus by \cite[Proposition 9]{DF}, to prove that $X$ is a spectral space, it suffices to show that  $d^{-1}(U)$ is a quasicompact subset of ${\ff X}$ for every quasicompact open subset $U$ of $X$.  Let $U$ be a basic  open subset of $X$.  Then there exist $f_1,\ldots,f_n \in F$ such that $U = \{x \in X: f_1,\ldots,f_n \in \OO_{X,x}\}$, and hence $d^{-1}(U) = \{V \in {\ff X}:f_1,\ldots,f_n \in V\}$ is a quasicompact open subset of ${\ff X}$. More generally, if $U$ is a quasicompact open subset of $X$, then $U = U_1 \cup \cdots \cup U_m$ for basic open subsets $U_i$ of $X$, and it follows that $d^{-1}(U)$ is a quasicompact subset of $\X$.  Therefore, $X$ is a spectral space.  
Suppose finally that $\Spec(D)$ is a Noetherian space.  
  The projective model $X$ of $F/D$ is covered by finitely many affine submodels $\Spec(D_i)$, $i=1,\ldots,m$, where each $D_i$ is a finitely generated $D$-subalgebra of  $F$.  A finitely generated algebra over a ring with Noetherian prime spectrum again has Noetherian prime spectrum \cite[Corollary 2.6]{OP}, and so  for each $i$, $\Spec(D_i)$ is a Noetherian space.  
It follows that $X$, as a finite union of Noetherian subspaces,  is a Noetherian space. % Moreover, since each $\Spec(D_i)$ is a Noetherian space, so  is each $\Spec(\overline{D}_i)$  \cite[Corollary, p.~372]{HFC}, and hence $\overline{X}$ is a Noetherian spectral space.   
\end{proof}

%Thus when $\Spec(D)$ is a Noetherian space, then  a
% subset $U$ of $X$ is  open if and only if $U$ is a finite union of basic open sets $\{x \in X: f_1,\ldots,f_n \in \OO_{X,x}\}$, where 
% $f_1,\ldots,f_n \in F$. 

   \begin{remark} \label{after prop}  {\em 
   Viewing ${\ff X}$ and $X$ as spectral spaces, the proof of the proposition shows that the map $d:{\ff X} \rightarrow X$ is a spectral map.  That $d$ is a spectral map was first observed by Dobbs, Fedder and Fontana, who  gave  a topological proof that when 
  $X = \Spec(D)$ and $D$ has quotient field $F$, then 
${\ff X}$ is a spectral space and $d$ is a spectral map  \cite[Theorem 4.1]{DFF}.  Dobbs and Fontana presented a sharper version of this  result in  \cite[Theorem 2]{DF}  by exhibiting a ring $R$ (namely, the Kronecker function ring of $R$ with respect to the $b$-operation, which  is discussed in the next section) such that ${\ff X}$ is homeomorphic to $\Spec(R)$.  That the Zariski-Riemann space $\X$ of $F/D$  is a spectral space when $D$ does not necessarily have quotient field $F$ follows from \cite[Proposition 2.7]{HK} or \cite[Corollary 3.6]{FFL}.  
  In the appendix of \cite{Kuh}, Kuhlmann gives a model-theoretic proof of the fact that ${\ff X}$ is a spectral space. 
  }
 \end{remark}

 When $X$ is a projective model of $F/D$ and $\Spec(D)$ is a Noetherian space, then by Proposition~\ref{Noetherian spectral}, $X$ is a Noetherian space, so that by Corollary~\ref{very new cor},  $\inv(Y) = \gen(Y)$ for every subset $Y$ of $X$.  
 %For this reason, the inverse topology  is too coarse to be of  much use  for  projective models over such a choice of $D$. 
 By contrast, subspaces of ${\ff X}$ are  generally   
  not quasicompact, and the inverse topology is more nuanced for this spectral space, as we see throughout the rest of the article.   We will see also that the patch topology  is more subtle on $\X$ than on the projective model $X$, the latter having been described in Proposition~\ref{patch Noetherian}.

   We say a collection ${\ff S}$ of projective models of $F/D$ is a  {\it dominant system} if for each pair of projective models $X$ and $Y$ of $F/D$ there exists a projective model $Z \in {\ff S}$ that dominates both $X$ and $Y$.  The set of all projective models of $F/D$ is a dominant system \cite[Lemma 6, p.~120]{ZS}.  On the other hand, if $F$ is finitely generated over the quotient field of $D$, then the set of projective models of $F/D$ having function field $F$ is also a dominant system.  In this second case, the notion of a dominant system gives us a convenient way to restrict to the birational setting.

 Let ${\ff S}$ be a dominant system of projective models. 
 %Whenever $X,Y \in {\ff S}$ and $Y$ dominates $X$, then, as discussed above, the domination morphism $\delta^Y_X:Y \rightarrow X$ is defined. Thus
 Since  each pair of projective models in  ${\ff S}$ is dominated by another in ${\ff S}$, then ${\ff S}$  forms a projective system in the category of locally ringed spaces.    Moreover, for each $X \in {\ff S}$, we have the domination morphism $\delta^{\ff X}_X:\X \rightarrow X$. 
 While not stated in this terminology, most of the following proposition is implicit in Zariski-Samuel \cite[Theorem VI.41, p.~122]{ZS}. 
 %We let $\OO_{\X}$ denote the sheaf on $\X$ defined one each nonempty open subset $U$ of $\X$ by $\OO_{\X}(U) = A(U) = \bigcap_{V \in U}V$.

% With these morphisms, $\X$ is in the projective limit of the projective system ${\ff S}$:

\begin{proposition} \label{lr} {\em (Zariski-Samuel)}  Let ${\ff S}$ be a dominant system of projective models of $F/D$.
\begin{itemize}

\item[{\em (1)}] As a topological space, $\X$ is the projective limit in the category of topological spaces of the underlying topological spaces of the projective models  in $ {\ff S}$.  

\item[{\em (2)}] As a spectral space, $\X$ is the projective limit in the category of spectral spaces of the underlying spectral spaces of the projective models  in $ {\ff S}$. Moreover, for each $X \in {\ff S}$, the map of spectral spaces  underlying the domination morphism $\delta^\X_X$ is a closed spectral map.

\item[{\em (3)}]  Each $V \in \X$ is the union of the $\OO_{X,x_V}$, where $X$ ranges over ${\ff S}$ and $x_V$ is the center of $V$ in $X$.  

\item[{\em (4)}] As locally ringed spaces,  $\X$ is the projective limit of the projective models in  ${\ff S}$.

\end{itemize}  
\end{proposition}

\begin{proof} Statements (1) and (3)  follow 
from the proof of Theorem VI.41, p.~122, in \cite{ZS}. 
Details for the proof of  statement (4) can  be found in \cite[Theorem 2.1.5]{Khan}.  To see that (2) holds, 
note that  the domination morphisms $\delta^Y_X:Y \rightarrow X$, where $Y$ and $X$ are projective models in ${\ff S}$ and  $Y$ dominates $X$, are of finite type, and hence are quasicompact \cite[Exercise II.3.2, p.~91]{Hart}. Thus the underlying continuous maps of the domination morphisms between projective models in ${\ff S}$ are spectral maps. It follows then from (1) that  $\X$ is the projective limit in the category of spectral spaces of the underlying spectral spaces of the projective models  in $ {\ff S}$. That the underlying map of the domination morphism $\delta^\X_X$ is a closed spectral map was discussed above. 
\end{proof}

   When $X$ is a projective model of $F/D$ and ${{Z}}$ is a subset of ${\ff X}$ we set \begin{center}
   $X({{Z}}) =$ the image of ${{Z}}$ in $X$ under the domination morphism $\delta^{\X}_X$.
   \end{center} Thus a point $x \in X$ is in $X({{Z}})$ if and only if there is   a valuation ring $V$ in $Z$ centered on $x$.

%
%there is an isomorphism of locally ringed spaces: $$\delta=(d,d^\#):{\ff X} \rightarrow \varprojlim {\ff S}$$. 
%
%The map $d:\X \rightarrow {\ff S}$ is defined by $d(V) = (x_V)
% $d(V) = (x_V)
%V \mapsto (\OO_{X,x_V})_{X \in {\ff S}},$$ where $x_V$ is the center of $V$ in $X$ and   $\varprojlim{\ff S}$ is viewed as a subset of the product $\prod_{X \in {\ff S}}X$. 

%In particular, each $V \in {\ff X}$ is the union of the local rings $\OO_{X,x_V}$, where $X$ ranges over ${\ff S}$.     

% In later sections we wish to single out the subscheme of a projective model $X$ for which the center of a given valuation ring in $X$ is the generic point.  When  $X$ is a projective model of $F/D$ and $V \in {\ff X}$ has center $x$ in $X$, we say that $\cl(\{x\})$ is the {\it closed subspace of $X$ defined by $V$}.   It follows from the proposition that if $X_V$ denotes the closed subspace of $X$ defined by $V$, then 
%  the closed subspace $\{W \in {\ff X}:W \subseteq V\}$ is the projective limit of the $X_V$, where $X$ ranges over ${\ff S}$.   

     %   The first lemma shows that the domination mapping commutes with the operators $\inv$, $\patch$ and $\pt$.   

\begin{corollary}  \label{fundamental cor}
Suppose $\Spec(D)$ is a Noetherian space.  Let ${\ff S}$ be a dominant system of projective models of $F/D$ and let  $Z$ be a subspace of $\X$.  Then:
\begin{eqnarray*}
 \inv(Z) & = &\{V \in \X:\forall X \in {\ff S}, {\mbox{the center of }} V {\mbox{ in }} X {\mbox{ specializes to a point in }}X(Z)\} \\
%$ is the set of all $V  \in {\ff X}$ such that  for every  $X \in {\ff S}$, the center of $V$ in $X$ specializes to a point in $X(Z)$.  
 \patch(Z) &  = &\{V \in \X:\forall X \in {\ff S}, \: V {\mbox{ is centered on a generic point of a subset of }} X(Z)\} \\
% is the set of all $V \in {\ff X}$ such that for every $X \in {\ff S}$, $V$ is centered in $X$ on a generic point for a subset of $X(Z)$.
%$\patch(Z)$ is the set of all $V \in {\ff X}$ such that for every $X \in {\ff S}$, $V$ is centered in $X$ on a generic point for a subset of $X(Z)$.  
  \pt(Z) & = & \{V \in \X:\forall X \in {\ff S}, \: V {\mbox{ is centered on a closed point in }} X(Z)\}.    
%\sharp(Z) & = & \{V \in \X:\exists X \in {\ff S}, \: {\mbox{the center of }} V {\mbox{ does not specialize to a point}} \\
%\: & \: & \:\:\:\:\:\:\:\:\:\:\:\:\:\:\:\:\:\:\:\:\:\:\:\:\:\:\:\:\:\:\:\:\:\:\:\:\:\:\:\:\:\:\:\:\:\:\:\:\:\:\:\:\:\:\:\:\:\:\:\:\:\:\:\:\:\:\:\:\:\:\:\:\:\:\:\:\:\:\:\:\:\:\:\:\:\:\:\:\:\:\:\:\:\:\:\:\:\:\:\:\:\:\:\:\:\:\:\:{\mbox{ in }} X(Z \smallsetminus \{V\})\}.
%\: & \: & 
%\:\:\:\:\:\:\:\:\:\:\:\:\:\:\:\:\:\:\:\:\:\:\:\:\:\:\:\:\:\:\:\:\:\:\:\:\:\:\:\:\:\:\:\:\:\:\:\:\:\:\:\:\:\:\:\:\:\:\:\:\:\:\:\:\:\:\:\:\:\:\:\:\:\:\:\:\:\:\:\:\:\:\:\:\:\:\:{\mbox{ to a point in }} X(Z \smallsetminus \{V\})\}.
\end{eqnarray*}

%\begin{itemize}

%\item[{\em (1)}] $\inv(Z)=\{V \in \X:\forall X \in {\ff S} {\mbox{ the center of }} V {\mbox{ in }} X {\mbox{ specializes to a point in }}X(Z)\}.$

%$ is the set of all $V  \in {\ff X}$ such that  for every  $X \in {\ff S}$, the center of $V$ in $X$ specializes to a point in $X(Z)$.  

%\item[{\em (2)}] $\patch(Z)=\{V \in \X:\forall X \in {\ff S}, \: V {\mbox{ is centered on a generic point of a subset of }} X(Z)\}$.

% is the set of all $V \in {\ff X}$ such that for every $X \in {\ff S}$, $V$ is centered in $X$ on a generic point for a subset of $X(Z)$.

%$\patch(Z)$ is the set of all $V \in {\ff X}$ such that for every $X \in {\ff S}$, $V$ is centered in $X$ on a generic point for a subset of $X(Z)$.  

%\item[{\em (3)}]  $\pt(Z) = \{V \in \X:\forall X \in {\ff S}, \: V {\mbox{ is centered on a closed point in }} X(Z)\}$.   

%\item[{\em (4)}] 
%$\sharp(Z) = \{V \in \X:\exists X \in {\ff S}, \: V {\mbox{ is centered on a point in  }} X {\mbox{ that does not specialize}}$ 

%$\:\:\:\:\:\:\:\:\:\:\:\:\:\:\:\:\:\:\:\:\:\:\:\:\:\:\:\:\:\:\:\:\:\:\:\:\:\:\:\:\:\:\:\:\:\:\:\:\:\:\:\:\:\:\:\:\:\:\:\:\:\:\:\:\:\:\:\:\:\:\:\:\:\:\:\:\:\:\:\:\:\:\:\:\:\:\:\:\:\:\:\:\:\:\:\:\:\:\:\:\:\:\:\:\:\:\:\:\:\:\:\:\:\:{\mbox{ to a point in }} X(Z \smallsetminus \{V\})\}$.

%\end{itemize}

 \end{corollary}

\begin{proof}
Since $\Spec(D)$ is a Noetherian space, each projective model of $F/D$  is by Proposition~\ref{Noetherian spectral} a Noetherian spectral space. Thus we may apply Theorem~\ref{fundamental} and Proposition~\ref{lr}(3)    to obtain the corollary. 
\end{proof}

  An example shows that without the assumption in the theorem that $\Spec(D)$ is a Noetherian space, the characterization of $\inv(Z)$ need not be valid.      

\begin{example}  {\em Let $R$ be a Pr\"ufer domain with quotient field $F$, and suppose that $R$ has a maximal ideal $M$ such that $\bigcap_{N \ne M}R_N \subseteq R_M$, where $N$ ranges over the maximal ideals of $R$ distinct from $M$. Examples of such Pr\"ufer domains include the ring of entire functions \cite[Proposition 8.1.1(6), p.~276]{FHP}, holomorphy and Kronecker function rings of function fields of transcendence degree $>1$ (\cite[Theorem 4.7]{OlbH} and \cite[Theorem 4.3]{HK}, respectively),
and the ring of integer-valued polynomials (e.g., combine \cite[Theorem 1.6]{GHmanu} and the proof of Proposition VI.2.8 in \cite{CC}.)  Then $X:=\Spec(R)$ is a projective model of $F/R$.  Set ${{Z}}_2 = \{R_N:N \in \Max(R) {\smallsetminus} \{M\}\}$ and ${{Z}}_1 = \Max(R)$.  Then $\bigcap_{V \in Z_1}V = \bigcap_{V \in Z_2} V$, and 
since the fact that $R$ is a Pr\"ufer domain implies that $\X$ is an affine scheme (see Section 4), then $\inv(Z_1) = \inv(Z_2)$ (see Proposition~\ref{top prelim}(5)).  Thus $R_M \in \inv(Z_1) = \inv(Z_2)$, but  
 the  closure of the center $M$ of the valuation ring $R_M \in {{Z}}_1$ in $X$  is simply $\{M\}$, which does not meet $X({{Z}}_2)$, in contrast to the characterization in Corollary~\ref{fundamental}.}
\end{example}

%When $Z$ is nonempty, then so are $\inv(Z)$, $\patch(Z)$ and $\pt(Z)$.  

\begin{remark} {\em  Suppose $D$ is a field and $F$ is a finitely generated field extension of $D$.  Let $Z \subseteq \X$.   Favre has shown that for any valuation ring $V \in \cl(Z)$, either $V \in \cl(\{U\})$ for some $U \in Z$ or there is a sequence $\{V_i\}_{i = 1}^\infty$ of valuation rings in $Z$ such that $V$ is the limit in the Zariski topology on $\X$ of the $V_i$ \cite[Theorem 3.1]{Fav}.  Thus every valuation ring in $\cl(Z) \cap \gen(Z)$ is a limit of a (countable) sequence of valuation rings in $Z$.  It follows also from \cite[Lemma 2.4]{Fav} that in the terminology of Remark~\ref{uf remark}(1),  every valuation ring in $\cl(Z) \cap \gen(Z)$ is an ultrafilter limit of countably many valuation rings in $Z$.  }
\end{remark}

%We assume, as in the last section, that $D$ is a subring of the field $F$ and $\Spec(D)$ is a Noetherian space, so that every projective model of $F/D$ is a Noetherian space.    
 % Since closure under generalizations is a transparent operation on the subspaces $Z$ of ${\ff X}$, the proposition shows that  
  % the main obstacle to finding $\inv(Z)$ is calculating the patch closure of $Z$.   

 In Corollary~\ref{Kuhlmann cor} we prove  patch density  of some canonically chosen collections of valuation rings.   The corollary is a consequence of a fact about the ``degenerate'' case for the intersection of valuation rings in a subspace of $\X$.  Recall that  
a projective model $X$  having function field $F$ is {\it normal} 
if $X$  is defined by $f_0,\ldots,f_n\in F$  such that the rings $D_i=D[\frac{f_0}{f_i},\ldots,\frac{f_n}{f_i}]$ in the resulting open cover 
  $X = \bigcup_{i=0}^n\Spec(D_i)$ are all integrally closed in $F$.   
  
  %The next proposition uses a version of 
 % Lemma~\ref{pick up extra} specialized to projective models:
  
  %   Using the proposition, we describe next a degenerate case for the ring $A({{Z}})$:

% \begin{lemma} \label{pick up extra2} \label{proj patch cor}   Let $X$ be a projective model  of $F/D$ and   ${{Z}}$ be a subset  of ${\ff X}$. Then each point in $X$ that is a generic point for a subset of $X({{Z}})$ has a valuation ring in $\patch({{Z}})$ centered on it.
%\end{lemma}  

%\begin{proof} Since the domination map $Z \rightarrow X$ is a closed spectral map, the  
% lemma follows immediately from Lemma~\ref{pick up extra}.  \end{proof}

\begin{proposition} \label{Hilbert} Suppose $D$ is a finitely generated algebra over a field.  If 
 $X$ is a normal projective model of $F/D$ having function field $F$, and ${{Z}}$ is a subset of ${\ff X}$  such that $X({{Z}})$ contains all but at most finitely many closed points of $X$, then    $A = \bigcap_{V \in Z}V$ is the integral closure of $D$ in $F$.  
\end{proposition}

\begin{proof}
There exist $f_1,\ldots,f_n \in F$  such that with $D_i = D[\frac{f_1}{f_i},\ldots,\frac{f_n}{f_i}]$, then each $D_i$ is an integrally closed Noetherian domain with quotient field $F$ and  $X= \bigcup_{i}\Spec(D_i)$.  Fix $i$, and  let  ${\ff p}$ be a height $1$ prime ideal of $D_i$.  Then since $D_i$ is a Hilbert domain, ${\ff p}$ is an intersection of  infinitely many maximal ideals of $D_i$.  Now since all but at most finitely many maximal ideals of $D_i$ are in $X({{Z}})$, it follows that the point in $X$ corresponding to ${\ff p}$ is a generic point of a subset of $X$.  
    Therefore, by Proposition~\ref{patch Noetherian},  ${\ff p}$ is in $\patch(X({{Z}}))$.  
    Since the domination map $Z \rightarrow X$ is a closed spectral map, Lemma~\ref{pick up extra}(4) implies that  each point in $X$ that is a generic point for a subset of $X({{Z}})$ has a valuation ring in $\patch({{Z}})$ centered on it.
    Thus since $(D_i)_{\ff p}$ is a DVR, it must be that $(D_i)_{\ff p} \in \patch({{Z}}) \subseteq \inv({{Z}})$.  This is true for every choice of $i$ and height $1$ prime ideal ${\ff p}$ of $D_i$, so since each $D_i$ is a Krull domain, we conclude  that $A \subseteq D_1 \cap \cdots \cap D_n$.  Moreover, every valuation ring in ${\ff X}$ contains at least one of the rings $D_i$.  Thus $A$ is the intersection of all valuation rings in ${\ff X}$, so that $A$ is the integral closure of $D$ in $F$.     
  \end{proof}

Thus with the assumptions on $D$, if for each closed point $x$ in $X$, the fiber ${{Z}}_x$ of the domination mapping $\delta:Z \rightarrow X$ is nonempty, then $A$ is  the integral closure of $D$ in $F$.  We use this observation to prove from a different point of view a density result due to Kuhlmann.  His result is stronger than what is stated here, and is a consequence of a powerful existence theorem for valuations on function fields.

%In \cite{Kuh}, Kuhlmann proves a number of existence results for valuations in algebraic function fields, many of which lead to  statements about density in the patch topology.  We collect a few of these in the next example.

\begin{corollary} \label{Kuhlmann cor} {\em (Kuhlmann \cite[Theorem 9]{Kuh})}
Let $D$ be a subfield of $F$, and suppose that $F/D$ is a finitely generated field  extension of transcendence $n\geq 1$.   
 Let  $d \geq 0$ and $r \geq 1$ such that $n \geq d+r$, and let ${{Z}}_{d,r}$ be the set of all  discrete valuation rings of rank $r$ whose residue field is a finitely generated field extension of $D$ having transcendence degree $d$.  Then ${{Z}}_{d,r}$ is patch dense in $\X$.  
\end{corollary}   

\begin{proof}  Let $X$ be a projective model of $F/D$.  Then since $X = \bigcup_{i=1}^n\Spec(D_i)$, where each $D_i$ is a Hilbert domain, it follows as in the proof of  Proposition~\ref{Hilbert} that the patch closure of the set of Zariski closed points of $X$ is $X$.  Thus 
in light of Lemma~\ref{patch thm}, it suffices to show that when $X$ is a projective model of $F/D$ with function field $F(X)$, then every closed point in $X$ has a valuation ring in ${{Z}}_{d,r}$ centered on it.   This is a standard fact about the valuation theory of function fields; for example, it follows from    \cite[Chapter VI, {\S}10.3, Theorem 1]{B} (but see the version stated in \cite[Lemma 2.6]{Kuh2}).  
\end{proof}
  
As a particular case of the corollary, the set ${{Z}}$ of  discrete rank one valuation rings (DVRs) $V$ in $\X$ whose  residue fields are finite over $k$ is patch dense in $\X$.  This has the interesting consequence that every valuation ring in $\X$ is an ultrafilter limit of valuation rings in ${{Z}}$; see Remark~\ref{uf remark}(1) below.  These DVRs in ${{Z}}$ are particularly tractable, since they arise from prime ideals in the generic formal fiber of local rings of closed points in projective models of $F/k$; see \cite[2.6, p.~25]{HRS}.           

In the context of the corollary, when   $k$ is a perfect field, then the   set ${{Z}}$  of DVRs in ${\ff X}$ having residue field $k$  lie dense with respect to the patch topology in the space ${{Z}}'$ of all valuation rings in ${\ff X}$ having residue field $k$ \cite[Corollary 5]{Kuh}.  If also $k$ is not algebraically closed, then as noted in Example~\ref{affine examples}, ${{Z}}$, and hence ${{Z}}'$, are ``affine'' sets, as defined in the next section. 

%We end with a few remarks on other instances of the patch topology in the literature.

\begin{remark} \label{uf remark} {\em 
\begin{itemize}
\item[]

\item[{(1)}]   The 
 patch closure in ${\ff X}$ has a  helpful interpretation in terms  of ultrafilter limits of valuation rings.  This is developed in \cite[Corollary 3.8]{FFL}, where it is shown that when ${{Z}}$ is a subset of ${\ff X}$, then $\patch({{Z}})$ is the set of valuation rings $V$ in ${\ff X}$ of the form $$V = \{x \in F:\{V \in {{Z}}:x \in V\} \in {\cal U}\},$$ where ${\cal U}$ is an ultrafilter on the set ${{Z}}$.

\item[{(2)}]  A  version of the Zariski-Riemann space has been developed for graded valuation rings  by  Temkin in \cite{TemNon}, where it is shown  that the graded Zariski-Riemann space is quasicompact \cite[Lemma 2.1]{TemNon}. 
  In \cite[Theorem 1.5]{Duc}, Ducros shows that certain maps between graded Zariski-Riemann spaces are closed with respect to the inverse topology.       
     The patch topology on the graded Zariski-Riemann  space is considered by Conrad and Temkin in  \cite[Lemma 5.3.6]{CT}.
 
\item[{(3)}] With $D = {\mathbb{C}}[[X,Y]]$, $F = {\mathbb{C}}((X,Y))$ and $\X_{0}$ the set of valuation rings in $F/D$ centered on the maximal ideal of $D$, Favre and Jonsson use  
the patch topology (the ``Hausdorff-Zariski'' topology) on ${\ff X}_0$  to describe features of  a tree structure on $\X_{0}$  whose partial ordering is  determined by behavior with respect to  sequences of blow-ups \cite[Proposition 5.29, p.~107]{FJ}.

\item[(4)]  Motivated by applications in real algebraic geometry and rigid analytic geometry,  Huber introduced the valuation spectra of a ring as a generalization of 
 the Zariski-Riemann surface of an extension $F/D$.  The patch topology is  often the topology of choice for valuation spectra; see Huber-Knebusch \cite{HubKne}.     
\end{itemize}      }
 \end{remark}

  \section{The Zariski-Riemann space as the image of an affine scheme}

 Again  we assume $D$ is a subring of a field $F$ but we do not assume any additional conditions (e.g., Noetherian) on $D$.  
  By Proposition~\ref{lr}, the locally ringed space ${\ff X}$  is a projective limit of projective schemes, but  is itself in general not a scheme.  However, by using the Kronecker function ring construction from multiplicative ideal theory, we can view ${\ff X}$ as the image of an affine scheme. 
We discuss how to do this in this section.  

Let  $T$ be an indeterminate for $F$.  For each valuation ring $V \in {\ff X}$, let $V^*$ be the Gaussian extension of $V$ to $F(T)$ ($=$ field of  rational functions in the variable $T$); that is, $$V^*  = V[T]_{{\ff M_V}[T]}.$$  Then $V^*$ is a valuation ring with quotient field $F(T)$ such that $V = V^* \cap F$. For a subset ${{Z}}$ of ${\ff X}$, we define the {\it Kronecker function ring of ${{Z}}$} to be the ring $$\Kr({{Z}}) = \bigcap_{V \in {{Z}}} V^*.$$  In the special case in which $D$ has quotient field $F$, then    $\Kr(\X)$ is the classical   
Kronecker function ring of $D$ with respect to the $b$-operation (see \cite[Section 26]{G}): $$\Kr(\X)=
\left\{\frac{f}{g}:f,g \in D[T], g\ne 0 {\mbox{ and }} \overline{c(f)} \subseteq \overline{c(g)}\right\},$$ where $\overline{I}$ denote the integral closure of the ideal $I$ in $D$ (which in the notation of Krull is  $I^b$, hence  the terminology of ``$b$-operation'').  Although we will not need it, this classical description of the Kronecker function ring of $\X$ with respect to $b$ can also be generalized to  projective space \cite{FH}.  

In any case, when ${{Z}}$ is a subspace of ${\ff X}$, then 
it follows from work of Halter-Koch on function rings that  $\Kr({{Z}})$ is a Pr\"ufer domain with quotient field $F(T)$; cf.~\cite[Corollary 3.6]{FFL}, \cite[Theorem 2.2]{HalKoc} and   
  \cite[Corollary 2.2]{HK}. Recall that a domain $R$ is a {\it Pr\"ufer domain} if every localization of $R$ at a prime ideal is a valuation ring; equivalently, every valuation ring between $R$ and its quotient field is a localization of $R$ at a prime ideal.  
   It follows then that if $R$ is a Pr\"ufer domain containing $D$ and having  quotient field $F$, then ${\ff X}_R \rightarrow \Spec(R)$ is an isomorphism of locally ringed spaces, and hence ${\ff X}_R$ has the structure of an  
  affine scheme.  Thus we obtain:  {\it For each subset  ${{Z}}$ of ${\ff X}$, the subspace ${\ff X}_{{\rm Kr}({{Z}})}$ of $\X^*:=\{V^*:V \in \X\}$ is an affine scheme.}

      We are interested also in when a subspace of $\X$ has the structure of an affine scheme.  We characterize these subspaces  in Theorem~\ref{affine scheme} as the inverse closed subspaces $Z$ of $\X$ such that $A = \bigcap_{V \in Z}V$ is a Pr\"ufer domain with quotient field $F$.  Motivated by this characterization,    
  we say a subset ${{Z}}$ of ${\ff X}$ is {\it affine} if $A=\bigcap_{V \in Z}V$ is a Pr\"ufer domain having quotient field $F$.  %It follows that  ${{Z}}$ is affine if and only if ${\ff X}_{A({{Z}})}$ is an affine integral scheme with function field $F$.  
  This  is a slight abuse of notation,
 since the possibly larger set $\inv(Z)$, but not ${{Z}}$ itself, is an affine scheme (in general, the affine set  ${{Z}}$ is only a subset of an affine scheme).  %However, %it follows from Proposition~\ref{top prelim}(5) that when ${{Z}}$ is affine, then  the inverse closure of ${{Z}}$ is an affine scheme.  
 Thus an affine set is  inverse dense  in an affine scheme.  

By the above remarks, there is  a canonical way to associate to each ${{Z}} \subseteq {\ff X}$ an affine set.   Define   $$ {{Z}}^*= \{V^*:V \in  {{Z}}\}.$$
Then, since $\Kr({{Z}}) = \bigcap_{V \in {{Z}}^*}V$, it follows that ${{Z}}^*$ is an affine subset of ${\X}^*=\{V^*:V \in {\ff X}\}$.  In Section~\ref{inverse section} we use this observation to describe subsets of ${\ff X}$ that are closed in the inverse topology.   
 By way of motivation, we mention here some examples of affine subsets of ${\ff X}$.

\begin{example}  \label{affine examples}
{\em 
\begin{itemize}
\item[]

\item[(1)]  Any finite subset of ${\ff X}$ is affine \cite[(11.11), p.~38]{N}.

\item[(2)]  If $D$ contains a field and ${{Z}} \subseteq {\ff X}$ has cardinality less than the cardinality of this field, then $A=\bigcap_{V \in Z}V$ is a Pr\"ufer domain, so that if also $A$ has quotient field $F$, then  ${{Z}}$ is affine \cite[Theorem 6.6]{OR}.  

\item[(3)]  If ${{Z}} \subseteq {\ff X}$ has the property that each valuation ring in ${{Z}}$ has a formally real residue field (i.e., $-1$ is not a sum of squares in the residue field), then ${{Z}}$ is an affine set. (The ring $A=\bigcap_{V \in Z}V$ is known in the literature of real algebraic geometry as the {\it real holomorphy ring} of $F/D$; see for example, \cite{Becker, Berr, BK, BS, Schu}.)  This is a special case of a more general phenomenon: If  there exists a nonconstant monic polynomial in $D[T]$ having no root in a residue field of any $V \in {{Z}}$, then ${{Z}}$ is affine.  This result has been proved independently and in various forms by several authors, including A.~Dress, R.~Gilmer, K.~A.~Loper and P.~Roquette; see \cite[p.~332]{OlbH} for precise references and a discussion of this result.   

\item[(4)]  If $D$ contains a field $k$ that is not algebraically closed and every valuation ring in ${{Z}} \subseteq {\ff X}$ has residue field $k$, then ${{Z}}$ is affine.  This is a special case of the general result in (3).  
\end{itemize}
}
\end{example}

   Returning to the case where ${{Z}} = {\ff X}$, we have (see \cite[Theorem 2.3]{HK}), 
    $$({\ff X}^*)_{{\rm Kr}(\X)} = \X^*.$$ %\:\: {\mbox{ and }} \:\: \Kr(F/H) = \bigcap_{V \in {\mathrm{Zar}}(F/H)}V^*.
    Thus since $\Kr(\X)$ is a Pr\"ufer domain, the stalks on $\Spec(\Kr(\X))$ are precisely the valuation rings in $\X^*$.   In this way, we obtain a  morphism of locally ringed spaces
    $\kappa=(k,k^\#):\X^* \rightarrow \X$ defined in the following way.  The map  $k:\X^* \rightarrow \X$ is defined by $k(V) = V \cap F$ for each $V \in \X^*$, and the morphism of sheaves $k^\#:\OO_{\X} \rightarrow k_*\OO_{\X^*}$ is defined for each nonempty open subset $U$ of $\X$  by $$k^\#_U:\OO_{\X}(U) \rightarrow \OO_{\X^*}(k^{-1}(U)):s \mapsto s$$ for each $s \in  \OO_{\X}(U)$. The map $k$ is a homeomorphism; see \cite{DF} for the affine case and  \cite[Corollary 3.6]{FFL} or \cite[Proposition 2.7]{HK}   for the general case.

     This accomplishes our first goal of mirroring  ${\ff X}$ with the affine scheme $\Spec(\Kr(\X))$, which in essence contains all of the valuation theory of $F/D$.  Although this affine scheme is not birationally equivalent to the object ${\ff X}$ (for it has function field $F(T)$ rather than $F$), we can change bases for each projective model $X$ of $F/D$ to produce a projective model $X^*$ of the extension $F(T)/D[T]$.  This is done via the {\it Nagata function ring}, $$D^*:= \left\{\frac{f}{g}:f,g \in D[T], g\ne 0 {\mbox{ and }} c(g) = D\right\},$$  where $c(g)$ denotes the content of the polynomial $g$. (The ring $D^*$ is often denoted $D(T)$, but to be consistent with our other notation, we use $D^*$ rather than $D(T)$.)  We let also $F^* = F(T)$.     
 The ring $D^*$
 is   a faithfully flat extension of $D$.  Therefore, 
  faithfully flat base extension produces the desired projective model: $$X^* = X \times_{{\textrm{Spec}}(D)} \Spec(D^*).$$  Thus when $X = \bigcup_{i=1}^n\Spec(D_i)$, we obtain that $X^*$ can be identified with $\bigcup_{i=1}^n\Spec(D^*[D_i]),$ and hence $X^*$ is a projective model of $F^*/D^*$. 
   In the next proposition, we collect many of the preceding observations into a single diagram.  
    
    \begin{proposition}\label{KFR}\label{KFR inverse} 
  With $X$ a projective model of $F/D$, there is a commutative diagram, 
  $$\begin{CD} \Spec(\Kr(\ff X)) & @>{\lambda}>> & \X^* & @>{\kappa}>> & {\ff X} \\
\:& \:& \: & \: & @VV{\delta^*}V & \: & @VV{\delta}V \\
\:& \:& \: & \: & X^* &  @>>{\pi}> & X,\end{CD}$$ where:
\begin{itemize}

\item[{\em (a)}]  $\lambda$ is the isomorphism of locally ringed spaces induced by localization:  $P \mapsto \Kr(\X)_P$ for each $P \in \Spec(\Kr(\X))$; 

\item[{\em (b)}] $\kappa:\X^* \rightarrow \X$ (which is defined above) is a morphism of locally ringed spaces that is a homeomorphism on the underlying topological spaces; 

\item[{\em (c)}] $\pi$ is the  surjective  morphism of schemes given by faithfully flat base extension;  

\item[{\em (d)}] $\delta$ is the domination morphism and hence is a closed surjective morphism of locally ringed spaces; and 

\item[{\em (e)}] $\delta^*$ is the  morphism of locally ringed spaces given by the restriction to $\X^*$ of the domination morphism 
from the Zariski-Riemann space of $F^*/D^*$ to $X^*$.  
   \end{itemize}
Each induced continous map in (a)--(e)  on the underlying topological spaces  is a spectral map, and hence is continuous in the Zariski, inverse and patch  topologies.   
%Moreover, the continuous maps induced by $\delta$ and $\delta^*$ are closed maps in the  patch topologies. 
  \end{proposition} 
 
 \begin{proof}
 Statements (a) and (b) were discussed above, and statement (c) is clear from the construction of $X^*$.  Statement (d) was discussed in Section~\ref{Zariski-Riemann section}, and statement (e) follows similarly. The commutativity of the  diagram is clear in light of the representation $X^* = \bigcup_{i=1}^n\Spec(D^*[D_i]).$  To see that all the maps are spectral maps, observe that 
 since the mappings $\lambda$ and $\kappa$ are homeomorphisms in the Zariski topology, it is clear that they are homeomorphisms in the inverse and patch topologies.  By Remark~\ref{after prop}, the continuous maps induced by $\delta$ and $\delta^*$ are spectral. 
  Moreover, it is clear that $\pi$ is continuous, and since $\pi$ is a finite type, hence  quasicompact, morphism,
  it induces a 
  spectral map on the underlying topological spaces. Thus   by Lemma~\ref{pick up extra}(1), this map is continuous 
 in the inverse and patch topologies.    
 \end{proof}

 \begin{remark} {\em The top row of the diagram in the proposition is also emphasized in \cite[Corollary 3.6 and Proposition 3.9]{FFL}.   
  That $\delta$ is continuous in the patch topology is proved in \cite[Proposition~3.9]{FFL} using an interpretation of patch closure   involving ultrafilters   (discussed in  Remark~\ref{uf remark}).   The argument there makes precise how $\delta$ preserves ultrafilter limits of valuation rings.   
}
 \end{remark}

%\begin{corollary} \label{pick up extra} If $\Spec(D)$ is a Noetherian space, $X$  is a projective model of $F/D$ and ${{Z}}$ is a subset of ${\ff X}$, 
% then   every point in $X$ that is a generic point for a subset of $X({{Z}})$ has a valuation ring in $\patch({{Z}})$ centered on it, and every in point $X$ that specializes to a point in $X$ has a valuation ring $\inv(Z)$ centered on it. 
%\end{corollary}

%\begin{proof}
%Since $\Spec(D)$ is a Noetherian space, $X$ is by Proposition~\ref{ } a Noetherian spectral space.  Thus the assertions follow from 

%The last statement follows from Proposition~\ref{patch Noetherian}.
%\end{proof}

   By the proposition, for each subset ${{Z}}$ of $ {\ff X}$, the subspace ${{Z}}^*$ of $\X^*$ is homeomorphic to ${{Z}}$ via the restriction of $\kappa$ to ${{Z}}^*$.  We rely on this observation  in the next section, since it allows us when considering topological properties of ${{Z}}$ to replace a non-affine subset of ${\ff X}$ with an affine subset of $\X^*$ homeomorphic to ${{Z}}$.

  \begin{corollary} \label{KFR2} Let $X$ be a projective model of $F/D$, and let ${{Z}}$ be a subspace of ${\ff X}$.  Then the continuous map $X^* \rightarrow X$ induced by the morphism $\pi:X^* \rightarrow X$ restricts to a homeomorphism $X^*({{Z}}^*) \rightarrow X({{Z}})$.  
  \end{corollary}
  
  \begin{proof}
  Since by Proposition~\ref{KFR} the domination mapping ${{Z}}^* \rightarrow \Spec(\Kr({{Z}}))$ is a homeomorphism, it follows from the  proposition that $X^*({{Z}}^*) \rightarrow X({{Z}})$ is onto.  To see that this mapping is one-to-one, suppose that $V,W \in {{Z}}$ such that $V$ and $W$ are centered on the same point of $X$, but $V^*$ and $W^*$ are centered on different points of $X^*$.  Write $X = \bigcup_{i}\Spec(D_i)$.  Then since $V$ and $W$ are centered on the same point of $X$, and $X^* = \bigcup_{i}\Spec(D^*[D_i])$,  it follows that there exists $i$ such that $V^*$ and $W^*$ are centered on different points of $\Spec(D^*[D_i])$.  Thus since $D^*[D_i] \subseteq D_i^*$, it must be that $V^*$ and $W^*$ are centered on different points of $D_i^*$, and there exist $f,g \in D_i[T]$ such that $c(g) = D_i$ and $f/g \in {\ff M}_{V^*} {\smallsetminus} {\ff M}_{W^*}$.  Then 
 $g/f \in W^*$, so that since $c(g) =D_i$, we have $1 \in c(g)W^* = gW^* \subseteq fW^* = c(f)W^*$.  Hence $c(f) \not \subseteq {\ff M}_{W^*}$.  On the other hand,  
  $fV^* = c(f)V^*$, so that $c(f) \subseteq  {\ff M}_V$.  Therefore,  $V$ and $W$ are centered on different points of $\Spec(D_i) \subseteq X$, a contradiction which implies the corollary.  
  \end{proof}

  In particular, the subspace $X^*(\X^*)$ of $X^*$ is homeomorphic to $X$.

  \begin{remark}  {\em  
  Let $k$ be a field, let $F$ be an extension of $k$, and let $L$ be an extension of $F$.  Denote by $\X_{F/k}$ and $\X_{L/k}$ the Zariski-Riemann spaces of $F/k$ and ${L/k}$, respectively.  In \cite[Theorem 1.4]{Duc}, Ducros shows that the map  $\X_{L/k} \rightarrow \X_{F/k}:V \mapsto V \cap F$ is a continuous map that is, in our terminology, closed with respect to the inverse topology.  Ducros also proves a graded version of this result \cite[Theorem 1.5]{Duc}.  Another version of functoriality is given in \cite[Proposition 4]{DF}, where it is shown that when $R \subseteq S$ is an extension of domains, then the Kronecker fuction ring of $R$ embeds in that of $S$.  
  }
  \end{remark}

%\begin{proposition}  If $ Z$ is a collection of valuation overrings  Noetherian-fibered over the Noetherian domain $D$ such that $\Spec_}(D) \subseteq \Max(D)$, then $H:=\bigcap_{V \in  Z}V$ is a Pr\"ufer domain.
%\end{proposition}

%\begin{remark}  {\em The proof of the theorem can be modified in a straightforward way to show that $V$ is a sharply closed point in ${{Z}}$ if and only if there exists a Noetherian spectral space $X$ and a spectral map $f:{\ff X} \rightarrow X$ such that $f({{Z}} {\smallsetminus} \{V\})$ is contained in an open set of $X$ not containing $f(V)$.}  
%\end{remark} 

\section{Inverse closed subspaces of $\X$}

\label{inverse section}

In this section we consider in more detail inverse closed subpaces of $\X$.
 We show first that these subspaces possess a locally ringed space structure inherited in an obvious way from $\X$ and that with this structure, they are expressible as a projective limit of locally ringed inverse closed subspaces of projective models.

Let $X$ be a set, and let $\{A_x:x \in X\}$ be a collection of quasilocal rings
 with $D \subseteq A_x \subseteq F$.  Then the {\it Zariski topology} on $X$ has as an open basis the sets of the form $\{x \in X:S \subseteq A_x\}$, where $S$ is a finite subset of $F$ \cite[Chapter VII, \S17]{ZS}. When 
  $X$ is a subset of $\X$  and $A_V = V$ for each $V \in X$, then this topology is precisely the subspace topology that $X$ inherits from $\X$. Similarly, when $X$ is a subset of a  projective model $Y$ of $F/D$  and $A_x = \OO_{Y,x}$ for each $x \in X$, then this topology  is the subspace topology on $X$.  
  
 In general, with $X$ a set and $\{A_x:x \in X\}$ a collection of quasilocal $D$-subalgebras over $F$, define a sheaf $\OO_X$ on $X$ by  
 % Let ${\cal X}$ be a locally ringed spectral space whose rings of sections (and hence whose stalks) are subrings of $F$. (Two examples we have in mind here are where ${\cal X}$ is a projective model of $F/D$ or ${\cal X}$ is the Zariski-Riemann space $\X$.)   
%We first distinguish which subspaces of ${\cal X}$ 
% are locally ringed spaces with a sheaf structure induced by the structure sheaf on ${\cal X}$. 
% Let $X$ be a subspace of ${\cal X}$. We define a presheaf $\OO_X$ on $X$ by defining
 $\OO_X(\emptyset) =  F$, and  
   for each nonempty open subset $U $ of $X$, $\OO_X(U) = \bigcap_{x \in U}A_x$, where for nonempty open sets $V \subseteq U$, the restriction map $\rho^U_V:\OO_X(U) \rightarrow \OO_X(V)$ is simply set inclusion. 
 Moreover, for each $x \in X$, the stalk $\OO_{X,x}$ of $\OO_X$ at $x$ is $A_x$.%\footnote{Let  $x \in X$. 

An inverse closed subset  of a projective model is a  spectral space (Proposition~\ref{Noetherian spectral}), as are the 
 inverse closed subsets of  $\X$ (Proposition~\ref{pick up extra}(1)).  
 Thus we deduce the following lemma from  the preceding discussion.

\begin{lemma} \label{discuss} If $X$ is a projective model of $F/D$, then every nonempty inverse closed subspace $Y$ of $X$ is a locally ringed spectral space with structure sheaf $\OO_Y$  defined above. Similarily, every nonempty inverse closed subspace $Z$ of $\X$ is a locally ringed spectral space with structure sheaf $\OO_Z$. \qed
\end{lemma}

%Therefore, when $X$ is a projective model of $F/D$, every nonempty inverse closed subspace $Y$ of $X$ is a locally ringed space with structure sheaf  $\OO_Y$ defined as above.  

In particular, when $Z$ is a subspace of $\X$, then $\inv(X(Z))$ is a locally ringed space. Implicit in the next theorem (and explicit in its proof) is the fact that when ${\ff S}$ is a dominant system of projective models of $F/D$, then $\{\inv(X(Z)):X \in {\ff S}\}$ is a projective system of locally ringed spaces.

\begin{theorem} \label{inverse case} Let $Z$ be a subspace of $\X$, and let ${\ff S}$ be  a  dominant system of projective models of $F/D$. Then, as locally ringed spaces, 
$$\inv({{Z}}) = \varprojlim\: \inv(X({{Z}})),$$
where $X$ ranges over ${\ff S}$.\end{theorem}

\begin{proof}  
For each $X \in {\ff S}$, let $X' = \inv(X(Z))$. By Lemma~\ref{discuss}, $X'$ is a locally ringed space with structure sheaf $\OO_{X'}$. 
 Let $X$ and $Y$ be projective models in ${\ff S}$ such that $Y$ dominates $X$, and let $\delta=(d,d^\#):Y \rightarrow X$ be the domination morphism. 
   Define a morphism $\phi=(f,f^\#):Y' \rightarrow X'$ in the following way. Let $f$ be the restriction of $d$ to $Y'$, so that $f$ is a continuous map from $Y'$ to $X$. We claim that $f(Y') \subseteq X'$. Now $d$ is a spectral map (for as a proper morphism, $\delta$ is quasicompact), and hence  by Lemma~\ref{pick up extra}(2),
     $f(Y') = d(\inv(Y(Z))) \subseteq \inv(d(Y(Z)))= \inv(X(Z)) = X'$, which shows that $f$ is a continuous map from $Y'$ into $X'$. 
     Next, define the sheaf morphism $f^\#:\OO_{X'} \rightarrow f_*\OO_{Y'}$ for each open set $U$ of $X'$ by $f^\#_U:\OO_{X'}(U) \rightarrow \OO_{Y'}(d^{-1}(U)):s \mapsto s$. This makes sense because  $\OO_{X'}(U) \subseteq\OO_{Y'}(d^{-1}(U))$. If $y \in Y'$, then the stalk of $Y'$ at $y$  is $\OO_{Y',y} = \OO_{Y,y}$, and from this it follows that  since $\delta$ is a morphism of locally ringed spaces, then so is $\phi$.    
   This shows that $\{X':X \in {\ff S}\}$ is a projective system of locally ringed spaces.    
  Now by Lemma~\ref{patch thm}, $\inv(Z)$ is the projective limit of the topological  spaces $X' = \inv(X(Z))$, where $X$ ranges over $ {\ff S}$.  
      Moreover, it follows from Proposition~\ref{lr}(3)  and the fact that the structure sheaf on each locally ringed space $X'$ is induced by the structure sheaf on $X$ that $\inv(Z)$ is the projective limit of $\{X':X \in {\ff S}\}$ in the category of locally ringed spaces.
 \end{proof}

\begin{remark}{\em 
Though it is a locally ringed space, an inverse closed subspace of a projective model of $F/D$ need not be a scheme. For example, when $D = k[T_0,T_1]$, with $k$ a field and $T_0$ and $T_1$ indeterminates for $k$, then the set $X$ of all prime ideals of $D$ of height $\leq 1$ is inverse closed in the projective model $\Spec(D)$ (since it is closed under generalizations and all open subsets of $\Spec(D)$ are quasicompact), but $X$ is not a scheme with respect to the structure sheaf $\OO_X$. For suppose $U$ is a nonempty open subset of $\Spec(D)$ such that $X \cap U$ is an affine scheme. Then 
$X \cap U$  contains all but at most finitely many height one prime ideals of $D$, and hence since $D$ is a UFD 
there is $f \in D$ such that $X \cap U = \{P \in \Spec(D):f \not \in P\}$. But then 
the ring of global sections of $X \cap U$  is  $D_f$, and this ring has Krull dimension $2$, while the stalks of $X \cap U$ all have dimension one. So $X$ cannot be a scheme.} 
\end{remark}

%\section{Inverse closure and Kronecker function rings}

 The  quasicompact open sets in ${\ff X}$ are of the form $\X_{D_1} \cup \cdots \cup \X_{D_n}$, where $D_1,\ldots,D_n$ are finitely generated $D$-subalgebras of $F$. 
 %and so the sets of this form constitute a basis for the inverse topology on $\X$.   
One of the advantages of working with affine subsets of ${\ff X}$ is that  it follows from Proposition~\ref{top prelim} that when  ${{Z}}$ is affine, then ${{Z}}$ is inverse closed if and only if ${{Z}} = \X_R$ for some ring $R$ with quotient field $F$.  That a subset of the form ${\ff X}_R$ is inverse closed is always true, but the converse, that an inverse  closed set has this form, requires additional hypotheses, such as  that ${{Z}}$ is affine.  For example, consider the case where  $D$ is a  local Noetherian UFD with quotient field $F$ and Krull dimension $>1$.  Let $p$ be a prime element of $D$.  Then ${\ff X}_{D[1/p]} \cup \X_{D_{(p)}}$ is an inverse closed subset of ${\ff X}$ that is not of the form $\X_R$ for any overring $R$ of $D$. As statement (3) of the next lemma shows, such examples do not arise for affine subspaces.

%The next lemma formalizes some of the technical properties  that make  affine subsets  more tractable.

%   When $Z$ is an affine subset of ${\ff X}$, then each valuation ring in $Z$ is a localization of the ring $A(Z)$ at a prime ideal.  Thus, as is stated more formally in Corollary~\ref{Prufer top lemma}, the valuation rings in $Z$ are stalks on an affine integral scheme; hence the terminology of ``affine'' subset.    

\begin{lemma} \label{qc open} Suppose $D$  is a Pr\"ufer domain  with  quotient field $F$.  

\begin{itemize}

\item[{\em (1)}] If $S$ and $T$ are $D$-submodules of $F$, then $\X_{S \cap T} = \X_S \cup \X_T$.  

\item[{\em (2)}]  A subset  of $\X$ is open and quasicompact in the Zariski topology if and only if it is of the form ${\ff X}_{S}$ for some finite subset $S$ of  $F$.  

\item[{\em (3)}]  A subset  of $\X$ is inverse closed in  $\X$ if and only if it is of the form $\X_S$ for some overring $S$.
\end{itemize}
\end{lemma}

\begin{proof}
(1) Let $A = S \cap T$, and let $V \in \X$.  Then since $D$ is a Pr\"ufer domain, $V = D_P$ for some prime ideal $P$ of $D$.  Hence $A_P = S_P \cap T_P$, and since $D_P$ is a valuation domain, this forces $A_P = S_P$ or $A_P = T_P$.  Thus  $S \subseteq V$ or $T \subseteq V$.

(2) Let ${{Z}}$ be a quasicompact open subset of $\X$.  Then since the subsets of $\X$ of the form $\X_S$, with $S$ a finite subset of $F$,  constitute  a basis of open subsets of $\X$ in the Zariski topology,   there exist finite subsets $S_1,\ldots,S_n$ of  $F$   such that ${{Z}} = \X_{S_1} \cup \cdots \cup \X_{S_n}$. For each $i=1,\ldots,n$, let $I_i$ be the fractional ideal of $D$ generated by $S_i$.  
Then by (1), $Z = \X_{I_1 \cap \cdots \cap I_n}$.  
%
% For each $k =1,\ldots,n$, there exists a finitely generated fractional ideal $I_k$ of $D$ such that $D_k = D[I_k]$.  Let $I = I_1 \cap \cdots \cap I_n$.  We show that ${{Z}} = \X_{I}$.  
% Since ${{Z}} = \X_{D_1} \cup \cdots \cup \X_{D_n}$ and $D[I] \subseteq D_1 \cap \cdots \cap D_n$, it is clear that ${{Z}} \subseteq \X_{I} $.  To prove the reverse inclusion, let $V \in \X_{I}$. Since $D$ is a Pr\"ufer domain, there exists a prime ideal $P$ of $D$ such that $V = D_P$.  Now $I_P = (I_1)_P \cap \cdots \cap (I_n)_P$, so since $D_P$ is a valuation domain, there exists $k$ such that $I_k \subseteq I_P \subseteq D[I]_P \subseteq V$.  Hence $D[I_k] \subseteq V$, which shows that $\X_{I} \subseteq \X_{D_1} \cup \cdots \cup \X_{D_n} = {{Z}}$.    
 Statement (2) now follows from the fact that since $D$ is a Pr\"ufer domain, the  intersection  of the finitely generated fractional ideals $I_1,\ldots,I_n$ is a finitely generated fractional ideal \cite[Proposition 21.4]{G}.

(3)  Suppose ${{Z}}$ is inverse closed in $\X$.  Then  by (2), ${{Z}}= \bigcap_{i}\X_{D_i}$ for a collection $\{D_i\}$ of overrings of $R$, and hence ${{Z}} = \X_{S}$, where  $S$ is the overring of $R$ generated by the rings $D_i$.  Conversely, if $S$ is an overring of $D$, then $\X_S = \bigcap_{s \in S}\X_{\{s\}}$, so that by (2), $\X_S$, as an intersection of inverse closed subsets, is inverse closed.          
\end{proof}

\begin{remark}  {\em   A domain $D$ with quotient field $F$ is {\it vacant} if it has a unique Kronecker function ring (for an explanation of the notion of a Kronecker function ring associated to a domain, see the discussion before Proposition~\ref{KFR closed}).  Fabbri has shown that these domains are characterized by a  version of  property (1) in the proposition:  $D$ is vacant if and only if $\X = \X_{D_1} \cup \cdots \cup \X_{D_n}$ whenever $D = D_1 \cap D_2 \cap \cdots \cap D_n$ \cite[Theorem 3.1]{Fab}.   It is an open question as to whether statement (1) is in general equivalent to the property of being vacant \cite[p.~1075]{Fab}.}
\end{remark}

The description of the inverse closure of non-affine subsets in ${\ff X}$ is less transparent, but  the Kronecker function ring construction discussed in the last section is useful here too in clarifying things.    %We introduce some additional notation in the proposition: If $R$ is a subring of $F^*$, then $$(\X^*_R)|_F = \{V \cap F:V \in \X^*_R\};$$ that is,  the %valuation rings in 
% $(\X^*_R)|_F$ are the restrictions of the valuation rings in $\X^*_R$ to $F$.      
Here, as in Section 3, $\X_R^*$ is the set of valuation rings in $\X^*$ containing $R$.  %Recall also that $A(Z) = \bigcap_{V \in Z}V$ for $Z \subseteq\X$.

 \begin{proposition} \label{top prelim lemma} \label{top prelim} Let $Z$ be a subspace of $\X$, and let $A = \bigcap_{V \in Z}V$. Then the following statements hold for $Z$ and $A$.
 %, and let $\inv(Z)$ denote the closure of $Z$ in the inverse topology.  
 \begin{itemize}
 \item[{\em (1)}] ${{Z}} = \inv({{Z}})$  if and only if ${{Z}}^* = \X^*_{\Kr({{Z}})}$. 
 \item[{\em (2)}]  $\inv({{Z}}) = \{V \cap F:V \in \X^*_{\Kr(Z)}\}$ %\X^*_{\Kr({{Z}})}$.
%  \{V \cap F: V \in \X^*_{\Kr({{Z}})}\}$.  
 
% \{V \cap F:V \in \Zar(\Kr(Z))\}$.
 \item[{\em (3)}]  $(\inv({{Z}}))^* = \X^*_{\Kr({{Z}})}$.
  \item[{\em (4)}] $A = \bigcap_{V \in {\rm{inv}}(Z)}V$. 
 \item[{\em (5)}] When ${{Z}}$ is affine, then $\inv({{Z}}) = \X_{A}$. 
\item[{\em (6)}]  $(\inv({{Z}}))^* = \inv({{Z}}^*)$.  
 \end{itemize}
 \end{proposition}
 
 \begin{proof}
(1) Suppose ${{Z}}$ is inverse closed in ${\ff X}$.  Then by Proposition~\ref{KFR inverse}, ${{Z}}^*$ is inverse closed in $\X^*$. Moreover, ${{Z}}^* \subseteq \X^*_{\Kr({{Z}})}$, so since $\Kr({{Z}})$ is a Pr\"ufer domain with quotient field $F^*$, we have by Lemma~\ref{qc open}(3) that  ${{Z}}^* = \X^*_R$ for some overring $R$ of $\Kr({{Z}})$. As an overring of a Pr\"ufer domain, $R$ is itself a Pr\"ufer domain, and in particular, it is the intersection of its valuation rings, so we conclude that $R  
 = \Kr({{Z}})$.   Therefore, ${{Z}}^* = \X^*_{\Kr({{Z}})}$.  
 Conversely, if ${{Z}}^* = \X^*_{\Kr({{Z}})}$, then ${{Z}}^*$ is by Lemma~\ref{qc open}(3) inverse closed in $\X^*$.  Thus by Proposition~\ref{KFR inverse}, ${{Z}}$ is inverse closed in  ${\ff X}$.  

(2) By Lemma~\ref{qc open}(3), $\X^*_{\Kr({{Z}})}$ is an inverse closed subset of $\X^*$, so by Proposition~\ref{KFR inverse}, $\{V \cap F:V \in \X^*_{\Kr(Z)}\}$ is an inverse closed subset of ${\ff X}.$ 
Thus to prove that $\inv({{Z}}) = \{V \cap F:V \in \X^*_{\Kr(Z)}\}$, it suffices to show that every inverse  closed subset of ${\ff X}$ containing ${{Z}}$ contains also $\{V \cap F:V \in \X^*_{\Kr(Z)}\}$.    
  If ${{Z}}'$ is another inverse closed subset of ${\ff X}$ containing ${{Z}}$, then by (1), $({{Z}}')^* = \X^*_{\Kr({{Z}}')}$.  Thus     $
  \{V \cap F:V \in \X^*_{\Kr(Z)}\} \subseteq \{V \cap F: V \in ({{Z}}')^*\} = {{Z}}'$.  Also, since ${{Z}} \subseteq {{Z}}'$, we have $\X^*_{\Kr({{Z}})} \subseteq \X^*_{\Kr({{Z}}')}$, so that $\{V \cap F:V \in \X^*_{\Kr(Z)}\} \subseteq {{Z}}'$.  
      This verifies (2). 
 
 (3) This follows from (2) and Proposition~\ref{KFR inverse}.  
 
 (4) It is enough by  (2) to observe that $\X^*_{\Kr({{Z}})}= \X^*_{\Kr(\inv({{Z}}))}$, which is the case by (3).  
 
(5) Since ${{Z}}$ is affine, $A$ is a Pr\"ufer domain with quotient field $F$.  Thus by Lemma~\ref{qc open}(3), $\inv({{Z}}) = \X_R$ for some ring $R$ between $D$ and $F$ that is integrally closed in $F$.  By (4), $A = \bigcap_{V \in {\rm{inv}}(Z)}V = \bigcap_{V \in \X_R}V= R$, so $\inv({{Z}}) = \X_{A}$.       
%By (3), $(\inv(Z))^* = \Zar(\Kr(Z))$.  But since $Z$ is affine, $\Zar(A(Z)) = \Zar_F(\Kr(Z))$ {\bf [ref]}.  Thus by (2), $\inv(Z) = \Zar(A(Z))$.

(6)  By (3) and (5), $\inv({{Z}}^*) = \X^*_{\Kr({{Z}})}= (\inv({{Z}}))^*$.
 \end{proof}

A consequence of the proposition is that  localizations of $A=\bigcap_{V \in Z}V$ can be represented with subsets of $\inv({{Z}})$.  

\begin{corollary} \label{localization} If ${{Z}}\subseteq {\ff X}$ and $S$ is a multiplicatively closed subset of $A=\bigcap_{V \in Z}V$, then  there exists $Y \subseteq \inv({{Z}})$ such that  $A_S = \bigcap_{V \in Y}V$. 
\end{corollary}

\begin{proof}
Since $A = \Kr({{Z}}) \cap F$, we have $A_S = \Kr({{Z}})_S \cap F$.  Moreover, $\X^*_{\Kr({{Z}})_S} \subseteq \X^*_{\Kr({{Z}})}$, so by Proposition~\ref{top prelim}(3), there is a subset $Y$ of $\inv({{Z}})$ such that $Y^* = \X^*_{\Kr({{Z}})_S}$.  Thus $\bigcap_{V \in Y}V = \Kr(Y) \cap F = \Kr({{Z}})_S \cap F = A_S$.  
\end{proof}

The proposition also yields an interpretation of $\pt(Z)$.

  \begin{corollary} \label{pt} Let ${{Z}}$ be a subspace of ${\ff X}$.  The continuous map, $$g:\Spec(\Kr({{Z}})) \rightarrow {\ff X}:P \mapsto \Kr({{Z}})_P \cap F$$  restricts to a homeomorphism of  $\Max(\Kr({{Z}}))$ onto $\pt({{Z}})$.
  \end{corollary}
  
  \begin{proof}  Let $M \in \Max(\Kr({{Z}}))$.  Then by Proposition~\ref{top prelim}(2), $g(M) = \Kr({{Z}})_M \cap F \in \inv({{Z}})$.  Moreover, the mapping $g$ is by Proposition~\ref{KFR} a  closed map, so since $M$ is a closed point in $\Max(\Kr({{Z}}))$, $g(M) \in \pt({{Z}})$.  Thus $g$ carries $\Max(\Kr({{Z}})))$ into $\pt({{Z}})$.  Now let  $V \in \pt({{Z}})$.  Then  Proposition~\ref{top prelim}(3) implies that $V^* \in \X^*_{\Kr({{Z}})}$, and hence there is a prime ideal $P$ of $\Kr({{Z}})$ such that $V^* = \Kr({{Z}})_P$.  By Proposition~\ref{KFR}, the mapping $g$ induces a homeomorphism of $\Spec(\Kr({{Z}}))$ onto $\{V \cap F:V \in \X^*_{\Kr({{Z}})}\}$. Since $V$ is a closed point in $\inv({{Z}})=\{V \cap F:V \in \X^*_{\Kr({{Z}})}\}$, then $V^*$ is a closed point in $\X^*_{\Kr({{Z}})}$, so that $P$ is a closed point in $\Spec(\Kr({{Z}}))$.  As such, $P \in \Max(\Kr({{Z}}))$, so that $g$ carries $\Max(\Kr({{Z}}))$ onto $\pt({{Z}})$. By Proposition~\ref{KFR}, $g$ is a closed continuous map, and the  proposition  follows
  \end{proof}

  We mention several  examples of inverse closed subsets.  Some explicit interpretations of $\inv(Z)$ are also given in \cite{OFuture} in the case where $D$ is a two-dimensional Noetherian domain with quotient field $F$.  

\begin{example} \label{formally real example} 
{\em \begin{itemize}
\item[]

\item[(1)]  
  Let $f(T)$ be a nonconstant monic polynomial in $D[T]$, and let ${{Z}}$ be the set of valuation rings $V$ in ${\ff X}$ such that $f$ has no root in the residue field of $V$.  If ${{Z}}$ is nonempty, then as noted in Example~\ref{affine examples}(2), ${{Z}}$ is affine.  Moreover, with $A = \bigcap_{V \in Z}V$, then  ${{Z}} = \X_{A}$.  For suppose $V \in \X_{A}$ and $f(x) \in {\ff M}_V$ for some $x \in V$. If $f(x) = 0$, then $x$ is in the integral closure of $D$ in $F$, and hence an element of each $V$ in $Z$, a contradiction to the fact that $f$ has no root in the residue field of $V$. Thus $f(x) \ne 0$ and since $f(x) \in {\ff M}_V$,  
   then $f(x)^{-1} \not \in V$, so that $f(x)^{-1} \not \in A$.  But then $f(x)^{-1} \not \in W$ for some $W \in {{Z}}$, so that since $W$ is a valuation ring, $f(x) \in {\ff M}_W$, a contradiction.  
Thus by   Proposition~\ref{top prelim}(5), ${{Z}} = \inv({{Z}})$. 

%\item[(2)]  
%Suppose that $F$ is a formally real field (see Example~\ref{affine examples}(2)), and let ${{Z}}$ be the subset of ${\ff X}$ consisting of the valuation rings that have formally real residue field. Then $Z = \bigcap_{i=1}^\infty Z_i$, where for each $i$, $Z_i$ is the collection of valuation rings whose residue field has the property that  $-1$ is not a sum of $i$ squares. Then by (1), ${{Z_i}} = \inv({{Z}})$.  

\item[(2)]  Suppose that $D$ is a Krull domain with quotient field $F$, and let ${{Z}}$ be the set of essential prime divisors of $D$, i.e., ${{Z}}$ is the set of all localizations of $D$ at height $1$ prime ideals of $D$.    Then $\Kr({{Z}})$ is a Dedekind domain with quotient field $F^*$, and hence $\X^*_{\Kr({{Z}})} ={{Z}}^*\cup\{F^*\}$.
Thus by Proposition~\ref{top prelim}(3), $\inv({{Z}}) = {{Z}} \cup \{F\}$. Moreover, if $D$ has Krull dimension $>1$, then ${{Z}}$ is not affine.      This example is a special case of the more general fact  that the inverse closure of a Noetherian subspace ${{Z}}$ of ${\ff X}$ is the closure of ${{Z}}$  under generalizations (Corollary~\ref{very new cor}).  
\end{itemize}}
\end{example}

It follows from  Proposition~\ref{top prelim}(5) that two {affine} subspaces ${{Z}}_1$ and ${{Z}}_2$ in ${\ff X}$ have the same inverse closure if and only if $\bigcap_{V \in Z_1}V = \bigcap_{V \in Z_2}V$. 
For non-affine subspaces we can assert via Proposition~\ref{top prelim}(4) only that when ${{Z}}_1$ and ${{Z}}_2$ have the same inverse closure, then $\bigcap_{V \in Z_1}V = \bigcap_{V \in Z_2}V$.   
  The reason that the converse of this is generally not true is that an integrally closed domain can have more than one Kronecker function ring associated to it. To clarify this, we say that for a ring $R$ between $D$ and $F$ that is integrally closed in $F$, an overring $S$ of $\Kr(\X)$  is a {\it Kronecker $F$-function ring} of $R$ if $R = S \cap F$. (This is a special case of the class of  $F$-function rings introduced by Halter-Koch in \cite{HalKoc}).   
 Thus $\Kr(\X_R)$ is the smallest Kronecker $F$-function ring of $R$, and when $R$ is a Pr\"ufer domain with quotient field $F$,  it is the unique Kronecker $F$-function ring of $R$ \cite[Theorem 26.18]{G}.  
 In general the property of having a unique Kronecker $F$-function ring  does not characterize Pr\"ufer domains   \cite{Fab}.  In any case, the next proposition shows that   the complexity (or lack thereof) of the class of   Kronecker $F$-function rings of a given domain $R$  explains the complexity of the class of inverse closed subspaces of ${\ff X}$ that represent $R$. 

\begin{proposition} \label{KFR closed} Let $R$ be a ring between $D$ and $F$ that is integrally closed in $F$.  Then there is a one-to-one correspondence between the Kronecker $F$-function rings $S$ of $R$  and inverse closed subspaces ${{Z}}$ of ${\ff X}$ such that $R = \bigcap_{V \in Z}V$. The correspondence is  given by $$S \mapsto \{V \cap F: V \in \X^*_S\} \:\:\:{\mbox{ and }} \:\:\: {{Z}} \mapsto \Kr({{Z}}).$$  
\end{proposition}

\begin{proof} Let ${{Z}}$ be a closed subspace of ${\ff X}$ such that $R = \bigcap_{V \in Z}V$.  Then 
by Proposition~\ref{top prelim}(2), $Z = \{V \cap F:V \in \X^*_{\Kr({{Z}})}\}$, so $\Kr({{Z}})$ is a Kronecker $F$-function ring of $R$.  Also, if $S$ and $T$ are Kronecker $F$-function rings of $R$ with $\{V \cap F:V \in \X^*_S\} = \{V \cap F:V \in \X^*_T\}$, then   $S = \bigcap_{V \in \X_S}V = \Kr(\{V \cap F:V \in \X^*_S\}) = \Kr(\{V \cap F:V \in \X^*_T\}) = \bigcap_{V \in \X_T}V = T$. %(see \cite[Theorem 26.15]{G}).  
This proves  the proposition.
\end{proof}

By way of example, every Krull domain of dimension at least $2$ has more than one Kronecker function ring; see \cite[Proposition 2.10]{Fab}. 

\section{Affine schemes in $\X$}

In this section we characterize the subspaces of $\X$ that are affine schemes.  As discussed at the beginning of Section~\ref{inverse section}, whether a subspace of $\X$ when equipped with the natural choice of a structure presheaf is a locally ringed subspace can be detected topologically, in the sense that this presheaf is a sheaf if and only if the  subspace is irreducible. 
In contrast, whether a subspace is an affine scheme cannot be detected topologically since $\X$ is homeomorphic to the affine scheme $\X^*$, regardless of the choice of $D$ and $F$.  Our first characterization of affine schemes in $\X$ is the following theorem.  
In the theorem, $\OO_Z$ is the presheaf defined at the beginning of Section~\ref{inverse section} as $\OO_Z(U) = \bigcap_{V \in U}V$ for all nonempty open subsets $U$ of $Z$.

      \begin{theorem} \label{affine scheme}
     Let $Z$ be a subspace   of $\X$. 
    Then $(Z,\OO_Z)$  is an affine scheme if and only if $Z$ is inverse closed in $\X$ and $\OO_Z(Z)$ is a Pr\"ufer domain with quotient field $F$.
      \end{theorem}     
      
      \begin{proof}
Let $A = \OO_Z(Z) = \bigcap_{V \in Z}V$, and suppose that $(Z,\OO_Z)$ is an affine scheme.    Then by assumption, $A_P \in Z$ for each prime ideal $P$ of $A$.  In particular, each localization of $A$ at a prime ideal is a valuation domain with quotient field $F$, and hence $A$ is a Pr\"ufer domain with quotient field $F$.  Moreover, if $V \in Z$, then since $Z$ is an affine scheme,  $V=A_P$ for some prime ideal $P$ of $A$,  so that $Z = \X_{A}$.  By Proposition~\ref{top prelim}(1), 
 $\X_{A}$ is inverse closed.  

Conversely, suppose that $A$ is a Pr\"ufer domain with quotient field $F$ and $Z$ is inverse closed. By Proposition~\ref{inverse lrs}, $Z$ is irreducible, and hence $\OO_Z$ is a sheaf.   
Since $A$ is a Pr\"ufer domain, then  each $V \in Z$ is a localization of $A$ at a prime ideal. Also, by Proposition~\ref{top prelim}(4), $Z = \inv(Z) = \X_{A}$, so it follows that each localization of $A$ at a prime ideal is in $Z$.   Thus $Z$ can be identified with $\Spec(A)$, and $(Z,\OO_Z)$ is an affine scheme.  
      \end{proof}

\begin{corollary} \label{ZR Prufer cor}The Zariski-Riemann space $\X$ is an affine scheme if and only if the integral closure of $D$ in $F$ is a Pr\"ufer domain with quotient field $F$. \qed
\end{corollary}

%\begin{proof} This follows from the theorem and the fact that the integral closure of $D$ in $F$ is $\OO_\X(\X)$ and has quotient field $F$.  
%\end{proof}

More generally, the inverse closed subsets of $\X$ correspond to affine schemes in $\X^*$:

\begin{corollary} A nonempty subset $Z$ of $\X$ is inverse closed if and only if $(Z^*,\OO_{Z^*})$  is an affine scheme.
\end{corollary}

\begin{proof}
If $Z$ is inverse closed, then by Proposition~\ref{KFR}, $Z^*$ is also inverse closed, and hence by Proposition~\ref{top prelim}(3), $Z^* = \X^*_{\Kr(Z)}$, so that by Corollary~\ref{ZR Prufer cor}, $(Z^*,\OO_{Z^*})$ is an affine scheme. Conversely, if $(Z^*,\OO_{Z^*})$  is an affine scheme, 
%Then since for any $V \in Z$, $V^*$ is a stalk of $Z^*$, and since $(Z^*,\OO_{Z^*})$ is a scheme, then $F^*$, as a localization of $V^*$, is in $Z^*$, and hence $Z^*$ is irreducible.  Therefore, 
then by
Theorem~\ref{affine scheme}, $Z^*$ is inverse closed, and hence by Proposition~\ref{KFR}, $Z$ is inverse closed.  
\end{proof}

Motivated by the terminology of \cite[Chapter VII, \S17]{ZS}, we define an {\it affine model} $X$ of $F/D$ to be an affine integral scheme of finite type over $D$ whose function field is a subfield of $F$; equivalently, there is a finitely generated $D$-subalgebra $R$ of $F$ such that $X = \{R_P:P \in \Spec(R)\}$.

\begin{corollary} \label{lim cor}  A nonempty subset $Z $ of $\X$ is affine if and only if $\inv(Z)$ is the projective limit of a projective system of affine models of $F/D$. 
\end{corollary}

\begin{proof}
Suppose that $Z$ is affine. Then by definition, $A=\bigcap_{V \in Z}V$ is a Pr\"ufer domain with quotient field $F$. Since by Proposition~\ref{top prelim}(4), $A = \bigcap_{V \in {\rm{\inv(Z)}}}V$, then by Theorem~\ref{affine scheme}, $\inv(Z)$ is 
 an affine scheme, so that by Proposition~\ref{top prelim}(5), the morphism $\inv(Z) \rightarrow \Spec(A)$ that sends a valuation ring in $\inv(Z)$ to its center on $A$ is an isomorphism of schemes. Since also
 $A$ is the direct limit of finitely generated $D$-subalgebras, it follows that $\inv(Z)$ is the projective limit of the projective system of affine models dominated by $\inv(Z)$.  Conversely, a projective limit of affine schemes is necessarily affine; see for example \cite[Lemma 01YW]{stacks}. 
\end{proof}

%The observation in Proposition~\ref{top prelim lemma}(4) that $A(Z) = A({\inv(Z)})$ can be found in 
%\cite[Proposition 5.1]{FFL}.

%In this section we assume that $D$ is a subring of a field $F$, and we introduce the stronger assumption, in force for the rest of the paper, that $\Spec(D)$ is a Noetherian space.  Thus by Proposition~\ref{Noetherian spectral}, projective models over $F/D$ are Noetherian spaces.  
%  We
%consider  closed and sharply closed points in $\inv({{Z}})$ when ${{Z}}$ is a subspace of ${\ff X}$.  These two classes of points play key roles in later sections.      
% Recall from Section 2 that when ${{Z}}$ is a subspace of a spectral space, then  $\pt({{Z}})$ is the set of  closed points of $\inv({{Z}})$.      Thus when ${{Z}} \subseteq{\ff X}$, then 
%   $\pt({{Z}})$ is the set of valuation rings in $\inv({{Z}})$ that are minimal with respect to inclusion. 

%It follows from Proposition~\ref{top prelim}(5) that when ${{Z}}_1$ and ${{Z}}_2$ are affine, then $A({{Z}}_1) = A({{Z}}_2)$ if and only if $\inv({{Z}}_1) = \inv({{Z}}_2)$.    
%   In light of Proposition~\ref{KFR closed}, the assumption of affineness is important here, because if  $R$ is a ring between $D$ and $F$ that is integrally closed in $F$, and $R$ has more than one Kronecker $F$-function ring, then there exists more than one inverse closed subspace of ${\ff X}$ that represents $R$, in contrast to the affine case.  

%
%However, by Proposition~\ref{top prelim}(5),  for affine spaces ${{Z}}$, $A({{Z}})$ completely determines $\inv({{Z}})$, and vice versa. 

We  give next a more refined version of Corollary~\ref{lim cor} by showing  that whether a subset $Z$ of $\X$ is affine is determined by whether every projective model is dominated by an affine model that is in turn dominated by $Z$. This is a consequence of the following routine lemma, whose proof we omit.  

\begin{lemma} \label{standard} Let $A$ be a domain with quotient field $F$, and let $t_0,\ldots,t_n$ be nonzero elements in $ F$. Then $(t_0,\ldots,t_n)A$ is an invertible fractional ideal of $A$ if and only if 
 $$A = \sum_{0=1}^n t_i:_A(t_0,\ldots,t_n)A.$$

\end{lemma}

%\begin{proof}
%Suppose that $(t_0,\ldots,t_n)A$ is an invertible fractional ideal of $A$.   To prove the equality of ideals in the statement of the lemma, it suffices to check the equality locally. Let $M$ be a maximal ideal of $A$. Then
%since $(t_0,\ldots,t_n)A_M$ is an invertible fractional ideal of $A_M$,  there exists $i$ such that $(t_0,\ldots,t_n)A_M = t_iA_M$, and hence $1 \in (t_i:_A (t_0,\ldots,t_n)A)A_M$.  
%  Conversely, to see that the formula implies $(t_0,\ldots,t_n)A$ is an invertible fractional ideal of $A$, it suffices to observe that for each maximal ideal $M$ of $A$, $(t_0,\ldots,t_n)A_M$ is a principal ideal of $A_M$. Let $M$ be a maximal ideal of $A$. Then from the formula in the lemma, we have $A_M = (t_i:_A (t_0,\ldots,t_n))A_M$ for some $i$, which in turn implies that $(t_0,\ldots,t_n)A_M = t_iA_M$.
%\end{proof}

\begin{theorem} \label{geom Prufer} The following are equivalent for a subset $Z$ of $\X$. 

\begin{itemize}

\item[{\em (1)}] $Z$ is affine.

\item[{\em (2)}] For every projective model $X$ of $F/D$, there exists an affine model $Y$ of $F/D$ dominating $X$ and dominated by $Z$. 

\item[{\em (3)}] For every projective model $X$ of $F/D$ defined by $2$ elements of $F$, there exists an affine model $Y$ of $F/D$  dominating $X$ and dominated by $Z$.

\end{itemize}
\end{theorem}

\begin{proof} (1) $\Rightarrow$ (2) 
Suppose that $Z$ is an affine subset of $\X$, and let $X$ be a projective model of $F/D$ defined by $t_0,\ldots,t_n \in F$ (as in Section 3). Let $A = \bigcap_{V \in Z}V$.  Then since $A$ is a Pr\"ufer domain, $(t_0,\ldots,t_n)A$ is an invertible fractional ideal of $A$, so we have  by Lemma~\ref{standard} that  $$A = \sum_{0=1}^n t_i:_A(t_0,\ldots,t_n)A.$$ Thus there exist $b_1,\ldots,b_n \in A$ and $a_{ij} \in A$, $i,j=1,\ldots,n$,  such that \begin{center} $(\dagger) \:\:\:\:\: \:\:\:\:\:\:\:\:\:\: 1 = b_1 + \cdots +b_n$ and for each $i$, $b_i = a_{i1}\frac{t_i}{t_1}=\cdots=a_{in}\frac{t_i}{t_n}.$\end{center}
 Define $R = D\left[\{b_1,\ldots,b_n\} \cup \{a_{ij}:i,j=1,\ldots,n\}\right].$ Then $R$ is a finitely generated $D$-subalgebra of $A$. Let $Y =\{R_P:P \in \Spec(R)\}$, so that $Y$ is an affine submodel of $F/D$. Since $R \subseteq A$, then $Z$ dominates $Y$. We claim that $Y$ dominates $X$. Indeed, let $P$ be a prime ideal of $R$.  Observe that from ($\dagger$) and the definition of $R$, it follows that  $R = \sum_{0=1}^n t_i:_R(t_0,\ldots,t_n)R.$ Thus there exists $i=1,\ldots,n$ such that 
 %$$ R \cap \frac{t_i}{t_1}R \cap \cdots \cap \frac{t_i}{t_n}R \not \subseteq P,$$ and hence 
 $R_P =  t_i:_{R_P}(t_0,\ldots,t_n)R_P.$ Consequently, $D[\frac{t_0}{t_i},\ldots,\frac{t_n}{t_i}] \subseteq  R_P$, and since the projective model $X$ is defined by $t_0,\ldots,t_n$, we conclude that $R_P$ dominates a local ring in $X$. Therefore, $Y$ dominates $X$.      
   
   (2) $\Rightarrow$ (3) This is clear. 
   
   (3) $\Rightarrow$ (1) 
   Assuming (3),   we claim that $A=\bigcap_{V \in Z}V$ is a Pr\"ufer domain with quotient field $F$. It suffices to show that for each prime ideal $P$ of $A$ and $0 \ne t \in F$, $t$ or $t^{-1} \in A_P$.  
    Let $P$ be a prime ideal of $A$ and let $0 \ne t \in F$. Let $X$ be the projective model of $F/D$ defined by $\{1,t\}$.  By assumption there is a finitely generated $D$-subalgebra $R$ of $F$ such that $Z$ dominates  the affine model $Y=\{R_Q:Q \in \Spec(R)\}$ (equivalently, $R \subseteq A$) while $Y$  dominates $X$.   
 Since $Y$ dominates  $X$ and $X$ consists of the localizations of $D[t]$ and $D[t^{-1}]$ at prime ideals, it follows that $R_{P \cap R}$, hence $A_P$, contains either $t$ or $t^{-1}$. Since the choice of $t$ was arbitrary, we conclude that $A_P$ is a valuation domain with quotient field $F$, and hence $A$ is a Pr\"ufer domain with quotient field $F$.
\end{proof}

\begin{corollary} A   domain $A$ with quotient field $F$ is a Pr\"ufer domain if and only if every projective model of $F/A$ is affine. 
\end{corollary}

\begin{proof} If $A$ is a Pr\"ufer domain and $X$ is a projective model of $F/A$, then $X$ consists of valuation rings between $A$ and $F$, and hence each member of $X$ is a localization of $A$. Also, if $V$ is a valuation ring between $A$ and $F$, then $V$ dominates $\OO_{X,x}$ for some $x \in X$. But $\OO_{X,x}$ is a valuation domain, so it follows that $V = \OO_{X,x} \in X$. Thus $X = \Spec(A)$. Conversely,   
if every projective model of $F/A$ is affine, then by Theorem~\ref{geom Prufer}, $A$ is a Pr\"ufer domain. 
\end{proof}

\smallskip

{\it Acknowledgement.} I thank Marco Fontana for pointing out to me a mistake in an earlier version of the article.

%%%%%%%%%%%%%%%%%%%%%%%%%%%%%%%%%%%%%%%%%%%%%%%%%%%%%%%%%%%%%%%%%
%%%%%%%%%%%%%%%%%%%%%%%%%%%%%%%%%%%%%%%%%%%%%%%%%%%%%%%%%%%%%%%%%
%%%%%%%%%%%%%%%%%%%%%%%%%%%%%%%%%%%%%%%%%%%%%%%%%%%%%%%%%%%%%%%%%

\end{document}